\documentclass[11pt]{article}
\usepackage{amsmath,amssymb,amsthm,amsfonts,amstext,amsbsy,amscd}
\usepackage[utf8]{inputenc}
\usepackage{amsfonts}
\usepackage{amssymb, amstext}
\usepackage{color}

\numberwithin{equation}{section}
\theoremstyle{plain}
\newtheorem{thm}{Theorem}

\newtheorem{corollary}{Corollary}
\newtheorem{lemma}{Lemma}
\newtheorem{assumption}{Assumption}

\theoremstyle{definition}

\theoremstyle{remark}
\newtheorem{remark}{Remark}

\newcommand{\SGN}{L}

\allowdisplaybreaks[4]

\begin{document}
\title{On the perturbation series for eigenvalues and eigenprojections}
\author{Martin Wahl\thanks{Humboldt-Universit\"{a}t zu Berlin, Germany. E-mail: martin.wahl@math.hu-berlin.de
\newline
\textit{2010 Mathematics Subject Classifcation.} 47A55, 15A42, 47H40, 62H25\newline
}}
\date{}
\maketitle
\begin{abstract}
A standard perturbation result states that perturbed eigenvalues and eigenprojections admit a perturbation series provided that the operator norm of the perturbation is smaller than a constant times the corresponding eigenvalue isolation distance. In this paper, we show that the same holds true under a weighted condition, where the perturbation is symmetrically normalized by the square-root of the reduced resolvent. This weighted condition originates in random perturbations where it leads to significant improvements.
\end{abstract}
\section{Introduction}

The study of perturbation bounds for eigenvalues and eigenprojections has a long tradition in matrix analysis and operator theory; see e.g. Horn and Johnson \cite{HJ94}, Bhatia \cite{B97}, and Chatelin \cite{C83}. In many application (including statistics, computer science, numerical analysis) it is crucial to quantitatively estimate how far eigenvalues and eigenprojections can move if the matrix or operator of interest is subjected to a perturbation.

In this paper, we are concerned with self-adjoint, compact operators $\Sigma$ and $\hat\Sigma$ on a Hilbert space $\mathcal{H}$. We consider $\hat \Sigma$ as an approximation of $\Sigma$, and define the perturbation operator $E=\hat\Sigma-\Sigma$ so that $\hat \Sigma=\Sigma+E$. By the spectral theorem, there is a sequence $(\lambda_1,\lambda_2,\dots)$ of eigenvalues of $\Sigma$ (converging to zero), together with an orthonormal system of eigenvectors $u_1,u_2,\dots$ such that $\Sigma =\sum_{j\geq 1}\lambda_j u_j\otimes u_j$. Similarly, there exists a sequence $(\hat{\lambda}_1, \hat{\lambda}_2,\dots)$ of eigenvalues of $\hat\Sigma$, together with an orthonormal system of eigenvectors $\hat{u}_1,\hat{u}_2,\dots$ such that $\hat{\Sigma}=\sum_{j\ge 1}\hat{\lambda}_j\hat u_j\otimes\hat u_j$. For every $j\geq 1$, we abbreviate $P_j=u_j\otimes u_j$ and  $\hat P_j=\hat u_j\otimes\hat u_j$. For the sake of simplicity, we assume throughout the introduction that $\dim \mathcal{H}=d<\infty$, meaning that all eigenvalues can be ordered as $\lambda_1\geq \dots\geq  \lambda_d$ and $\hat{\lambda}_1\geq \dots\geq  \hat{\lambda}_d$.

A standard result in perturbation theory states that the $j$-th perturbed eigenvalue $\hat \lambda_j$ and its corresponding perturbed eigenprojection $\hat P_j$ admit a Taylor series in the perturbation $E$, provided that the operator norm of the perturbation is smaller than $1/2$ times the eigenvalue isolation distance of the $j$-th unperturbed eigenvalue $\lambda_j$, that is provided that $\|E\|_\infty/g_j<1/2$ with $g_j=\min(\lambda_{j-1}-\lambda_j,\lambda_j-\lambda_{j+1})$, see e.g. Theorem 3.9 in Kato \cite{K95}. Quantitative versions of this statement, such as a $k$-th order Taylor expansion for $\hat P_j$ with an estimate for the remainder in terms of  $(\|E\|_\infty/g_j)^{k+1}$, are used throughout the statistical literature to analyze spectral algorithms including functional and kernel PCA; see e.g. Koltchinskii \cite{K98}, Mas and Menneteau \cite{MM03}, El Karoui and d'Aspremont \cite{KA10}, Hsing and Eubank \cite{HE15}, and Koltchinskii and Lounici~\cite{KL16b}.

In this paper, we investigate the conditions under which perturbed eigenvalues and eigenprojections admit accurate Taylor approximations. We show that the condition $\|E\|_\infty/g_j< 1/2$ can be replaced by a weighted version, where the perturbation is symmetrically normalized by the square-root of the reduced resolvent. Specialized to simple eigenvalues, this condition reads as
\begin{equation}\label{EqNormCond}
\delta_j:=\left\Vert\left(|R_j|^{1/2}+g_j^{-1/2}P_j\right)E\left(|R_j|^{1/2}+g_j^{-1/2}P_j\right)\right\Vert_\infty<1/2,
\end{equation}
where $R_j=\sum_{k\neq j}(\lambda_k-\lambda_j)^{-1}P_k$ is the reduced resolvent of $\Sigma$ at $\lambda_j$. Clearly, we have $\delta_j\leq \|E\|_\infty/g_j$, so that we introduce an extended framework under which perturbation problems can be attacked. Moreover, the quantity $\delta_j$ explicitly appears in the remainder error bounds, implying many classical and new perturbation bounds.

Condition \eqref{EqNormCond} originates in random perturbations where it leads to significant improvements over standard approaches. There is indeed a vast literature on the operator norm of (infinite) random matrices, implying that the conditions $\|E\|_\infty/g_j<1/2$ (as well as $\|(|R_j|+g_j^{-1}P_j)E\|_\infty<1/2$) are significantly stronger than their symmetrized variant $\delta_j<1/2$ from \eqref{EqNormCond}; see e.g. Lata\l a \cite{L05}, Koltchinskii and Lounici \cite{KL14}, and Adamczak, Lata\l a and Meller \cite{ALM18}.

As an application we will specialize our results to the empirical covariance operator in the case of i.i.d. and sub-Gaussian observations. In this case we will see that Condition \eqref{EqNormCond} is closely linked to an eigenvalue separation property, allowing us to study empirical eigenvalues and eigenprojections in a (nearly) optimal range.
For instance, in the case that the eigenvalues of $\Sigma$ decay exponentially, condition $\|E\|_\infty/g_j<1/2$ requires $j\leq c\log n$, while \eqref{EqNormCond} requires $j\leq cn$ (note that $\hat\Sigma$ is of rank $n$, hence all eigenprojections with index larger $n$ are non-unique). This large range will reveal a sharp phase transition, showing that the relative rank condition from \cite{JW18a,JW18b} is even necessary for an accurate first order perturbation expansion of $\hat\lambda_j$ in the case of i.i.d. and sub-Gaussian observations.

Improving standard perturbation results in the case of the empirical covariance operator has attracted interest recently. For instance, Mas and Ruymgaart \cite{MR15} combined the holomorphic functional calculus with a normalization argument to go beyond the condition  $\|E\|_\infty/g_j< 1/2$. Jirak and Wahl \cite{JW18a,JW18b} exploited the relative structure of the perturbation problem given by the empirical covariance operator, and proved general perturbation expansions for eigenvalues and eigenprojections in the relative rank setting. In contrast, we present a general perturbation-theoretic approach
that also leads to results beyond the conditions from \cite{JW18a,JW18b} when specialized to the empirical covariance operator. Our approach based on $\delta_j$ is similar to the one in Rei\ss{} and Wahl \cite{RW18} (where eigenprojections are studied in a different loss function, the so called excess risk), and extends their results to higher order expansions. Linear perturbation expansions based on $\delta_j$ are also used in \cite{JW19}, where quantitative limit theorems and bootstrap procedures are considered. Finally, improving absolute perturbation bounds based on $\|E\|_\infty$ is also subject in other branches of mathematics; see e.g. Ipsen \cite{I98} for relative perturbation bounds in numerical analysis, Belkin \cite{B18} for kernel operators used in machine learning problems, and Vu \cite{V11} and O'Rourke, Vu, and Wang \cite{RVW18} for random perturbations of low-rank matrices.

\subsection*{Further notation}
Let $(\mathcal{H},\langle\cdot,\cdot\rangle)$ be a separable Hilbert space of dimension $d\in\mathbb{N}\cup\{+\infty\}$ and let $\|\cdot\|$ denote the norm on $\mathcal{H}$, defined by $\|u\|=\sqrt{\langle u,u\rangle}$. As in the introduction, let $\Sigma$ be a self-adjoint, compact operator on $\mathcal{H}$, having spectral representation
\[
\Sigma =\sum_{j\ge 1}\lambda_j P_j,
\]
where $P_j=u_j\otimes u_j$. Here, for $u,v\in\mathcal{H}$ we denote by $u\otimes v$ the rank-one operator defined by $(u\otimes v)x=\langle v,x\rangle u$, $x\in \mathcal{H}$. For $j\geq 1$, let $g_j$ be the $j$-th spectral gap defined by $g_j=\min(\lambda_{j-1}-\lambda_j,\lambda_j-\lambda_{j+1})$ for $j\geq 2$ and $g_1=\lambda_1-\lambda_2$. Finally, for $j\geq 1$ such that $g_j>0$, the reduced resolvent at $\lambda_j$ is defined by
\[
R_j=\sum_{k\neq j}\frac{1}{\lambda_k-\lambda_j}P_k.
\]

Let $\hat\Sigma$ be another self-adjoint, compact operator on $\mathcal{H}$, having spectral representation
\[
\hat \Sigma =\sum_{j\ge 1}\hat \lambda_j \hat P_j,
\]
where $\hat P_j=\hat u_j\otimes \hat u_j$. We consider $\hat\Sigma$ as a perturbed version of $\Sigma$ and write $E=\hat\Sigma-\Sigma$ for the perturbation operator.

Since we want to compare perturbed and unperturbed eigenvalues and eigenprojections, we have to order eigenvalues accordingly. For the sake of notational simplicity, we assume throughout the paper that the eigenvalues of $\Sigma$ and $\hat\Sigma$ can be ordered in non-increasing order (meaning that $\lambda_1\geq \lambda_2\geq\dots$ and $\hat{\lambda}_1\geq \hat{\lambda}_2\geq\dots$), and that the orthonormal systems $u_1,u_2,\dots$ and $\hat u_1,\hat u_2,\dots$ are indeed orthonormal bases (meaning that $\sum_{j\geq 1}P_j=\sum_{j\geq 1}\hat P_j=I$). This imposes no restriction if $\dim \mathcal{H}=d<\infty$, while we restrict ourselves to positive operators if $d=\infty$ (to avoid ordering positive, negative and zero eigenvalues separately). This restriction is not essential since every statement can be obtained from the finite-dimensional case by approximation.

Given a bounded (resp. Hilbert-Schmidt) operator $A$ on $\mathcal{H}$, we write $\|A\|_\infty$ (resp. $\|A\|_2$) for the operator norm (resp. the  Hilbert-Schmidt norm). Given a trace class operator $A$ on $\mathcal{H}$, we denote the trace of $A$ by $\operatorname{tr}(A)$.

\section{Main results and some consequences}
\subsection{Error bounds for Taylor approximations}\label{SecTaylorExp}
Throughout this section, let $j\geq 1$ be such that $\lambda_j$ is a simple eigenvalue, meaning that $g_j>0$. Extensions to multiple eigenvalues are presented in Section \ref{SecMultEV}.

For every $n\geq 0$, set
\begin{equation*}
P_j^{(n)}=P_j^{(n)}(E)=(-1)^{n+1}\sum_{\substack{k_1,\dots,k_{n+1}\geq 0\\k_1+\dots+k_{n+1}=n}}R_j^{(k_1)}ER_j^{(k_2)}\cdots R_j^{(k_{n})}ER_j^{(k_{n+1})}
\end{equation*}
with $R_j^{(k)}=-P_j$ if $k=0$ and $R_j^{(k)}=R_j^k$ if $k>0$. For instance, we have
\[
P_j^{(0)}=P_j\quad\text{and}\quad P_j^{(1)}=-R_jEP_j-P_jER_j.
\]
It is well-known that the $P_j^{(n)}$ give the coefficients of a power series, valid in a small neighborhood of $\Sigma$. More precisely, letting $\Sigma(t)=\Sigma+t(\hat\Sigma-\Sigma)$ and $P_j(t)$ be the orthogonal projection onto the $j$-th eigenvector of $\Sigma(t)$, then the map $t\mapsto P_j(t)$ is well-defined and real analytic in a neighborhood of $0$, and we have $P_j^{(n)}= (1/n!)(d/dt)^n\mid_{t=0} P_j(t)$, see e.g. Chapter 2 in \cite{K95}.

We now introduce two quantities which will play a crucial role in what follows. First, let us recall $\delta_j$ from the introduction defined by
\[
\delta_j:=\Vert(|R_j|^{1/2}+g_j^{-1/2}P_j)E(|R_j|^{1/2}+g_j^{-1/2}P_j)\Vert_\infty.
\]
Second, a quantity more directly related to the $P_j^{(n)}$ is given by
\[
\delta_j':=\max(\Vert|R_j|^{1/2}E|R_j|^{1/2}\Vert_\infty,g_j^{-1/2}\Vert |R_j|^{1/2}E P_j\Vert_2,g_j^{-1}\Vert P_j E P_j\Vert_2).
\]
The two quantities $\delta_j'$ and $\delta_j$ are closely related. In fact, we have
\begin{equation}\label{EqEquivDefDelta}
\delta_j'\leq \delta_j\leq 2\delta_j',
\end{equation}
as can be seen by simple properties of the operator norm.

Our first main result is an error bound for the Taylor approximation of eigenprojections.

\begin{thm}\label{ThmTaylorSP}
For every $p\geq 1$, we have
\begin{equation}\label{EqTaylorSP}
\|\hat P_j-\sum_{n=0}^{p-1}P_j^{(n)}\|_2\leq   4g_j^{-1/2}\Vert P_jE|R_j|^{1/2}\Vert_2\frac{(4\delta_j')^{p-1}}{(1-2\delta_j)^2},
\end{equation}
provided that $\delta_j<1/2$.
\end{thm}
\begin{remark}
In the numerator $\delta_j'$ can be replaced by $\delta_j$, and in the denominator, $1-2\delta_j$ can be replaced by $1-4\delta_j'$, as can be seen by \eqref{EqEquivDefDelta}.
\end{remark}
\begin{remark}
Theorem \ref{ThmTaylorSP} yields $\|\hat P_j-\sum_{n=0}^{p-1}P_j^{(n)}\|_2\leq C {\delta_j'}^{p}$ with $C=4^{p+1}$, provided that $\delta_j\leq 1/4$. Moreover, the condition $\delta_j\leq 1/4$ can be dropped by increasing the constant $C$, cf. Corollary \ref{CorImprovedBound2} below. These bounds extend standard perturbation results from the statistical literature, bounding the left-hand side of \eqref{EqTaylorSP} by $C(\|E\|_\infty/g_j)^p$; see e.g. \cite{K98,HE15}.
\end{remark}

We now turn to the analysis of the eigenvalues. Set $\lambda_j^{(0)}=\lambda_j$ and for $n\geq 1$,
\[
\lambda_j^{(n)}=\operatorname{tr}(P_j^{(n-1)}E)+\operatorname{tr}(P_j^{(n)}R_j^{-1}).
\]
For, instance, we have
\[
\lambda_j^{(1)}=\operatorname{tr}(P_jEP_j)\quad\text{and}\quad \lambda_j^{(2)}=-\operatorname{tr}(P_jER_jEP_j).
\]
Our second main result is an error bound for the Taylor approximation of eigenvalues.
\begin{thm}\label{ThmTaylorEV}For every $p\geq 2$, we have
\begin{equation}
|\hat \lambda_j-\sum_{n=0}^{p-1}\lambda_j^{(n)}|\leq   12\Vert P_jE|R_j|^{1/2}\Vert_2^2\frac{(4\delta_j')^{p-2}}{(1-2\delta_j)^3},
\end{equation}
provided that $\delta_j<1/2$.
\end{thm}
\begin{remark}
The case $p=1$ can be deduced from the case $p=2$, see Corollary \ref{PertBoundEV} below.
\end{remark}

\subsection{Perturbation series under a symmetrically weighted condition}\label{SecPertSeries}
The following corollary shows that the condition $\delta_j'<1/4$ suffices to guarantee that the $j$-th perturbed eigenvalue and eigenprojection admit series representations in the perturbation $E$, cf. page 76 in \cite{K95} or \cite{C83}.
\begin{corollary}\label{CorPertSeries}
If $\delta_j'<1/4$, then the series of real numbers $\sum_{n=0}^\infty \lambda_j^{(n)}$  and the series of Hilbert-Schmidt operators $\sum_{n=0}^\infty P_j^{(n)}$ converge absolutely  and we have
\[
\hat \lambda_j=\sum_{n=0}^\infty \lambda_j^{(n)}\qquad\text{and}\qquad\hat P_j=\sum_{n=0}^\infty P_j^{(n)}.
\]
\end{corollary}
Corollary \ref{CorPertSeries} is a consequence of Theorems \ref{ThmTaylorSP} and \ref{ThmTaylorEV} (in combination with Lemmas \ref{LemNumberCoeff} and \ref{LemCoeffBound} below to get absolute convergence of the series). Note that we have to replace the condition $\delta_j<1/2$ by the slightly stronger assumption $\delta_j'<1/4$ in order to ensure that the error terms in Theorems \ref{ThmTaylorSP} and \ref{ThmTaylorEV} converge to zero.

Conversely, the perturbation series can be used to get estimates for remainder terms. This can be done by computing the number of terms in the definition of the $P_j^{(n)}$ and $\lambda_j^{(n)}$ combined with term by term bounds (often called enumerative method). While the following result basically follows from Theorems \ref{ThmTaylorSP} and \ref{ThmTaylorEV}, it is streamlined in terms of constants and follows by a simple application of Corollary \ref{CorPertSeries}.

\begin{corollary}\label{CorTaylorExpPS}
Suppose that $\delta_j'<1/4$. Then, for each $p\geq 1$, we have
\begin{equation*}
\|\hat P_j-\sum_{n=0}^{p-1}P_j^{(n)}\|_2\leq  4g_j^{-1/2}\Vert P_jE|R_j|^{1/2}\Vert_2\frac{(4\delta_j')^{p-1}}{1-4\delta_j'}\leq  \frac{(4\delta_j')^{p}}{1-4\delta_j'}.
\end{equation*}
Moreover, for each $p\geq 2$, we have
\begin{equation*}
|\hat \lambda_j-\sum_{n=0}^{p-1}\lambda_j^{(n)}|\leq 8\Vert P_jE|R_j|^{1/2}\Vert_2^2 \frac{(4\delta_j')^{p-2}}{1-4\delta_j'}\leq g_j \frac{(4\delta_j')^{p}}{1-4\delta_j'}.
\end{equation*}

\end{corollary}

\begin{remark}
A version of the second inequalities with $4\delta_j'$ replaced by $2\delta_j$ can be deduced from the holomorphic functional calculus (in combination with the eigenvalue separation from Lemma \ref{LemEVSeparation}, see Appendix \ref{SecHFC} for the details).
\end{remark}
\subsection{Tight perturbation bounds}\label{SecTightPB}
In this section, we use our main results to obtain some perturbation bounds for eigenvalues and eigenprojections. These bounds are close to optimal and go beyond perturbation bounds based on $\|E\|_\infty$, such as $|\hat\lambda_j-\lambda_j|\leq \|E\|_\infty$ and $\|\hat P_j-P_j\|_\infty\leq 2\sqrt{2}\|E\|_\infty/g_j$; see e.g. \cite{HJ94,B97}.

First, Theorem \ref{ThmTaylorEV} applied with $p=2$, the fact that $|\lambda_j^{(1)}|=|\operatorname{tr}(P_jEP_j)|= \|P_jEP_j\|_2$, and the triangular inequality yield the following perturbation bound for eigenvalues.
\begin{corollary}\label{PertBoundEV} If $\delta_j\leq 1/2-\epsilon$, $\epsilon\in(0,1/2)$, then there is a constant $C>0$ depending only on $\epsilon$ such that
\[
|\hat \lambda_j-\lambda_j|\leq \|P_jEP_j\|_2+C\Vert P_jE|R_j|^{1/2}\Vert_2^2.
\]
\end{corollary}
\begin{remark}
Corollary \ref{PertBoundEV} is close to optimal. It gives the absolute value of the linear perturbation term plus a remainder equals to $C\operatorname{tr}(P_jE|R_j|EP_j)$. The latter term differs from the quadratic perturbation term only by the absolute value of the resolvent.
\end{remark}
Concerning eigenprojections, Theorem \ref{ThmTaylorEV} applied with $p=1$ gives the perturbation bound $\|\hat P_j-P_j\|_2\leq Cg_j^{-1/2}\Vert |R_j|^{1/2}EP_j\Vert_2$. Yet, the linear term in the perturbation series is $-R_jEP_j-P_jER_j$, meaning that we would rather desire an upper bound $C\|R_jEP_j\|$. This discrepancy can be further removed by a more detailed analysis of higher-order perturbation expansions or the perturbation series.

\begin{corollary}\label{CorImprovedBound}
If $\delta_j'\leq 1/4-\epsilon$, $\epsilon\in(0,1/4)$, then there is a constant $C>0$ depending only on $\epsilon$ such that
\begin{equation*}
\|\hat P_j-P_j\|_2\leq C\sum_{m=1}^\infty\|(R_jE)^mP_j\|_2.
\end{equation*}
\end{corollary}
\begin{remark}
While the condition $\delta_j'<1/4$ implies a strong contraction property for perturbed eigenvalues (it gives a bound which merely includes the first and second perturbation terms), Corollary \ref{CorImprovedBound} still contains an infinite series.

Invoking $\delta_j'<1/4$, we again obtain Theorem \ref{ThmTaylorSP} with $q=1$. An interesting question is to determine under which assumptions the infinite sum in Corollary \ref{CorImprovedBound} can be reduced to $C\|R_jEP_j\|_2$. Simple proposals are (cf. \cite{K95,C83})
\[
\|R_j\|_\infty\|E \|_\infty\leq c<1\qquad\text{or}\qquad \||R_j|E\|_\infty\leq c.
\]
While these conditions seems in general comparable to the symmetrized variant $\Vert|R_j|^{1/2}E|R_j|^{1/2}\Vert_\infty<c$, we try to avoid them since they require significant stronger conditions in the case of random perturbations; see e.g. \cite{L05,KL14,vH15}. As an illustration of this phenomenon note in the case that $E\in \mathbb{R}^{p\times p}$ is a GOE matrix, both, $\mathbb{E}\|R_j\|_\infty\|E\|_\infty$ and $\mathbb{E}\||R_j|E\|_\infty$ are bounded by $C\|R_j\|_\infty\sqrt{p}$, while $\mathbb{E}\delta_j$ is bounded by $C\sqrt{\|R_j\|_\infty\operatorname{tr}(|R_j|)}$, as can be seen by applying \cite{vH15}. This leads to different conditions by using the Gaussian concentration property. Another illustration of this phenomenon is given in the case of the empirical covariance operator in Section \ref{SecAppli} below.
\end{remark}

The condition $\delta_j'< 1/4$ can be dropped by introducing an additional remainder term.

\begin{corollary}\label{CorImprovedBound2}
Suppose that $g_j>0$. For every natural number $p\geq 1$, there is a constant $C>0$ depending only on $p$ such that
\begin{equation*}
\|\hat P_j-P_j\|_2\leq C\sum_{m=1}^{p-1}\|(R_jE)^mP_j\|_2+C\delta_j'^{p}.
\end{equation*}
\end{corollary}
With a slightly more careful analysis, one can replace the sum of the norms by the norm of the sum.
\begin{corollary}\label{CorImprovedBound3}
Suppose that $g_j>0$. For every natural number $p\geq 1$, there is a constant $C>0$ depending only on $p$ such that
\begin{equation*}
\|\hat P_j-P_j\|_2\leq C\|\sum_{m=1}^{p-1}(R_jE)^mP_j\|_2+C\delta_j'^{p}.
\end{equation*}
\end{corollary}
\subsection{Extension to multiple eigenvalues}\label{SecMultEV}
Finally, we present an extension of Theorem \ref{ThmTaylorSP} to the case of multiple eigenvalues. Let $\mu_1>\mu_2>\dots>0$ be the sequence of positive and distinct eigenvalues of $\Sigma$. For $r\geq 1$, let $\mathcal{I}_r=\{j\geq 1:\lambda_j=\mu_r\}$. For $r\geq 1$, let $g_r=\min(\mu_{r-1}-\mu_r,\mu_r-\mu_{r+1})$, and let
\begin{equation}\label{SpectralProj}
P_r=\sum_{j\in \mathcal{I}_r}u_j\otimes u_j \quad\text{and}\quad \hat{P}_r=\sum_{j\in \mathcal{I}_r}\hat{u}_j\otimes \hat{u}_j.
\end{equation}
For $r\geq 1$, define the reduced resolvent
\[
R_r=\sum_{s\neq r}\frac{1}{\mu_s-\mu_r}P_s,
\]
and $R_r^{(k)}=-P_r$ if $k=0$ and $R_r^{(k)}=R_r^k$ if $k>0$. Then we have the following extension of Theorem \ref{ThmTaylorEV}.
\begin{thm}\label{ThmTaylorMultiple} Suppose that
\[
\delta_r:=\Vert(|R_r|^{1/2}+g_r^{-1/2}P_r)E(|R_r|^{1/2}+g_r^{-1/2}P_r)\Vert_\infty<1/4.
\]
Then there is an absolute constant $C>1$ such that for every $p\geq 1$,
\begin{align*}
&\|\hat P_r-\sum_{n=0}^{p-1}\sum_{\substack{k_1,\dots,k_{n+1}\geq 0\\k_1+\dots+k_{n+1}= n}}(-1)^{k_1+\cdots +k_{n+1}}R_r^{(k_1)}E\cdots ER_r^{(k_{n+1})}\|_2\leq C(4\delta_r)^p.
\end{align*}
\end{thm}
The proof of Theorem \ref{ThmTaylorMultiple} follows a similar but more tedious line of arguments as the proof of Theorem \ref{ThmTaylorSP}, and is therefore not presented in this paper.

\section{Proof of the  main results}\label{SecProof}
In this section, we prove Theorems \ref{ThmTaylorSP} and \ref{ThmTaylorEV}. The proof is based on the analysis of a Taylor expansion with explicit remainder term. Additionally, we present the proofs for the consequences from Sections \ref{SecPertSeries} and \ref{SecTightPB}.
\subsection{Preliminary lemmas}\label{SecPreLemmas}
The following simple lemma gives the number of terms in the formula for $P_j^{(n)}$.
\begin{lemma}\label{LemNumberCoeff}
The number of $(n+1)$-tuples $(k_1,\dots,k_{n+1})\in\mathbb{N}_{0}^{n+1}$ such that $k_1+\dots+k_{n+1}=m$ is equal to $\binom{n+m}{n}\leq 2^{n+m}$.
\end{lemma}
Our first crucial step is to show that the condition $\delta_j<1/2$ implies that the perturbed eigenvalues $\hat \lambda_{j-1},\hat \lambda_j,\hat \lambda_{j+1}$ are well separated.
\begin{lemma}\label{LemEVSeparation} If $\delta_j<1/2$ then we have
\begin{equation}\label{EqEVSep}
|\hat \lambda_j-\lambda_j|\leq \delta_jg_j
\end{equation}
as well as
\begin{equation}\label{EqEVSep2}
\hat \lambda_{j+1}-\lambda_{j+1}\leq \delta_j(\lambda_j-\lambda_{j+1}),\quad \hat \lambda_{j-1}-\lambda_{j-1}\geq -\delta_j(\lambda_{j-1}-\lambda_{j}).
\end{equation}
\end{lemma}

\begin{proof}
Set
\begin{align*}
T_{\geq j}=\sum_{k\geq  j}\frac{1}{\sqrt{\lambda_j+\delta_jg_j-\lambda_k}}P_k,\quad T_{\leq j}=\sum_{k\leq  j}\frac{1}{\sqrt{\lambda_k+\delta_jg_j-\lambda_j}}P_k.
\end{align*}
Then \cite[Proposition 1]{JW18b} states that $\hat\lambda_j-\lambda_j\leq \delta_jg_j$ (resp. $\hat\lambda_j-\lambda_j\geq -\delta_jg_j$), provided that $\|T_{\geq j}ET_{\geq j}\|_\infty\leq 1$ (resp. $\|T_{\leq j}ET_{\leq j}\|_\infty\leq 1$). Now, by simple properties of the operator norm, using that $\sqrt{\lambda_j+\delta_jg_j-\lambda_k}\geq \sqrt{\lambda_j-\lambda_k}$ for every $k>j$, we have
\begin{align*}
\|T_{\geq j}ET_{\geq j}\|_\infty
&\leq \Vert(|R_j|^{1/2}+(\delta_jg_j)^{-1/2}P_j)E(|R_j|^{1/2}+(\delta_jg_j)^{-1/2}P_j)\Vert_\infty\\
&\leq \delta_j^{-1}\Vert(|R_j|^{1/2}+g_j^{-1/2}P_j)E(|R_j|^{1/2}+g_j^{-1/2}P_j)\Vert_\infty\leq 1.
\end{align*}
By the above, we conclude that $\hat\lambda_j-\lambda_j\leq \delta_jg_j$. Similarly, we have $\|T_{\leq j}ET_{\leq j}\|_\infty\leq 1$, implying that $\hat\lambda_j-\lambda_j\geq  -\delta_jg_j$, and \eqref{EqEVSep} follows. Moreover, for
\begin{align*}
T_{>j}&=\sum_{k\geq j+1}\frac{1}{\sqrt{\lambda_{j+1}+\delta_j(\lambda_j-\lambda_{j+1})-\lambda_k}}P_k,\\
T_{< j}&=\sum_{k\leq j-1}\frac{1}{\sqrt{\lambda_k+\delta_j(\lambda_{j-1}-\lambda_j)-\lambda_{j-1}}}P_k,
\end{align*}
we have
\[
\|T_{< j}ET_{< j}\|_\infty,\|T_{>j}ET_{>j}\|_\infty\leq \delta_j^{-1}\Vert|R_j|^{1/2}E|R_j|^{1/2}\Vert_\infty\leq 1,
\]
and another application of \cite[Proposition 1]{JW18b} yields  \eqref{EqEVSep2}.
\end{proof}

We now state an explicit formula for the remainder term when approximating $\hat P_j$ with a $(p-1)$-th Taylor polynomial in $E$.
\begin{lemma}\label{LemExplicitRem} Suppose that $\delta_j<1/2$. Then, for every $p\geq 1$, we have
\begin{equation}\label{EqExplicitRem}
\hat P_j-\sum_{n=0}^{p-1}P_j^{(n)}=(-1)^{p-1}\sum_{k_1,\dots,k_p\geq 0}R_j^{(k_1)}E\cdots ER_j^{(k_p)}E\hat R_j^{(p-k_1-\dots-k_p)}
\end{equation}
with $\hat R_j^{(k)}=-(\hat\lambda_j-\lambda_j)^{-k}\hat P_j$ if $k\leq 0$ and $R_j^{(k)}=\hat R_j^k$ if $k>0$, where
\[
\hat R_j=\sum_{k\neq j}\frac{1}{\hat\lambda_k-\lambda_j}\hat P_k.
\]
\end{lemma}
\begin{remark}
By Lemma \ref{LemEVSeparation}, $\hat R_j$ is well-defined.  Moreover, the right-hand side in \eqref{EqExplicitRem} converges by Lemmas \ref{LemNumberCoeff} and \ref{LemEVSeparation}.
\end{remark}

\begin{proof}
We would like to establish \eqref{EqExplicitRem} by induction on $p$. For every $k\geq 1$, we have
\begin{equation}\label{EqBasicBB}
(\hat \lambda_j-\lambda_k)P_k\hat P_j=P_kE\hat P_j,\quad (\hat \lambda_k-\lambda_j)P_j\hat P_k=P_jE\hat P_k.
\end{equation}
Summing these identities over $k\neq j$ and using Lemma \ref{LemEVSeparation} yields
 \begin{align}
 P_j(I-\hat P_j)&=P_jE\hat R_j\label{EqExpStep1},\\
(I-P_j)\hat P_j&=\sum_{k\neq j}\frac{1}{\hat\lambda_j-\lambda_k}P_kE\hat P_j=-\sum_{l=1}^\infty (\hat\lambda_j-\lambda_j)^{l-1}R_j^{l}E\hat P_j.\label{EqExpStep2}
\end{align}
Hence,
\begin{align*}
\hat P_j-P_j&=(I-P_j)\hat P_j-P_j(I-\hat P_j)=-\sum_{l=1}^\infty (\hat\lambda_j-\lambda_j)^{l-1}R_j^{l}E\hat P_j-P_jE\hat R_j,
\end{align*}
which gives the claim for $p=1$, as can be seen by inserting the definition of $\hat R_j^{(k)}$. For the induction step assume that \eqref{EqExplicitRem} holds for $p$. First, the induction beginning can be written as
\begin{equation}\label{EqExpId1}
\hat R_j^{(0)}=R_j^{(0)}-\sum_{l=0}^\infty R_j^{(l)}E\hat R_j^{(1-l)}.
\end{equation}
Similarly, one can show that
\begin{align}\label{EqExpId2}
&\forall k>0,\quad\hat R_j^{(k)}=R_j^{(k)}-\sum_{l=0}^\infty R_j^{(l)}E\hat R_j^{(k+1-l)},\\
&\forall k<0,\quad \hat R_j^{(k)}=-\sum_{l=0}^\infty R_j^{(l)}E\hat R_j^{(k+1-l)}.\label{EqExpId3}
\end{align}
Letting $k=p-k_1-\cdots-k_p$, the claim follows from inserting \eqref{EqExpId1}-\eqref{EqExpId3} into \eqref{EqExplicitRem} and setting $l=k_{p+1}$. It remains to prove \eqref{EqExpId2} and \eqref{EqExpId3}. First, for $k<0$, we insert \eqref{EqBasicBB} and \eqref{EqExpStep1} to get
\begin{align*}
\hat R_j^{(k)}&=-(\hat\lambda_j-\lambda_j)^{-k}\hat P_j\\
&=-(\hat\lambda_j-\lambda_j)^{-k}P_j \hat P_j-(\hat\lambda_j-\lambda_j)^{-k}(I-P_j) \hat P_j\\
&=-(\hat\lambda_j-\lambda_j)^{-k-1}P_j E\hat P_j+\sum_{l=1}^\infty (\hat\lambda_j-\lambda_j)^{l-1-k}R_j^{l}E\hat P_j,
\end{align*}
which gives \eqref{EqExpId3} by inserting the definitions of $R_j^{(k)}$ and $\hat R_j^{(k)}$. Finally, for $k\geq 1$, we have
\begin{equation}\label{EqProofExpId}
\hat R_j^k=(I-P_j)\hat R_j^k+P_j\hat R_j^k.
\end{equation}
Using \eqref{EqBasicBB}, we have
\begin{equation*}
(I-P_j)\hat R_j= R_j(I-\hat P_j)-R_jE\hat R_j
\end{equation*}
and iterating this identity leads to
\begin{align*}
(I-P_j)\hat R_j^k&= R_j^k(I-\hat P_j)-\sum_{l=1}^kR_j^lE\hat R_j^{k+1-l}.
\end{align*}
Inserting this, \eqref{EqExpStep1} and \eqref{EqExpStep2} into \eqref{EqProofExpId}, we get
\begin{align*}
\hat R_j^k&=  R_j^k(I-\hat P_j)-\sum_{l=1}^kR_j^lE\hat R_j^{k+1-l}+P_j\hat R_j^k\\
 &=R_j^k+\sum_{l=k+1}^\infty (\hat\lambda_j-\lambda_j)^{l-k-1}R_j^{l}E\hat P_j-\sum_{l=1}^kR_j^lE\hat R_j^{k+1-l}+P_jE\hat R_j^{k+1}\\
 &=R_j^k-\sum_{l=0}^\infty R_j^{(l)}E\hat R_j^{(k+1-l)},
\end{align*}
which completes the proof of \eqref{EqExpId2}.
\end{proof}

\begin{lemma}\label{LemCoeffBound} Let $n\geq 1$ and let $(k_1,\dots,k_{n+1})\in\mathbb{N}_0^{n+1}$ with $k_1+\dots+k_{n+1}=m$. Then we have
\begin{equation*}
\|R_j^{(k_1)}ER_j^{(k_2)}\cdots R_j^{(k_{n})}ER_j^{(k_{n+1})}\|_2\leq g_j^{n-m}{\delta_j'}^{n}.
\end{equation*}
Moreover, if $k_a=0$ for some $a\leq n+1$ and if $m\geq 1$, then we have
\begin{equation*}
\|R_j^{(k_1)}ER_j^{(k_2)}\cdots R_j^{(k_{n})}ER_j^{(k_{n+1})}\|_2\leq g_j^{n-m}g_j^{-1/2}\||R_j|^{1/2}EP_j\|_2{\delta_j'}^{n-1}.
\end{equation*}
\end{lemma}

\begin{remark}\label{RemCoeffBound}
Analogous results hold for $R_j^{(k_1)}$ or $R_j^{(k_{n+1})}$ replaced by $|R_j|^{1/2}$.
\end{remark}

\begin{proof}
Set
\begin{equation}\label{EqDefSj}
S_j^{(k)}=P_j\text{ if $k=0$ and }S_j^{(k)}=|R_j|^{1/2} \text{ if $k>0$.}
\end{equation}
Let us focus on the case that $k_1,k_{n+1}\geq 1$, the other cases follow by similar arguments. First, we have
\begin{align}
&\|R_j^{(k_1)}ER_j^{(k_2)}\cdots R_j^{(k_{n})}ER_j^{(k_{n+1})}\|_2\nonumber\\
&\leq g_j^{-\sum_{a=1}^{n+1}(k_a-1)_+-1}\prod_{a=1}^{n}\|S_j^{(k_a)}ES_j^{(k_{a+1})}\|_\infty.\label{EqCoeffBound}
\end{align}
Using that all terms appearing in the product are of the form
\[
\Vert|R_j|^{1/2}E|R_j|^{1/2}\Vert_\infty,\quad \Vert |R_j|^{1/2}E P_j\Vert_\infty,\quad \Vert P_j E P_j\Vert_\infty
\]
and also that
\begin{equation*}
n+1-m+\sum_{a=1}^{n+1}(k_a-1)_+=|\{a:R^{(k_a)}=-P_j\}|,
\end{equation*}
we get
\begin{align*}
&\|R_j^{(k_1)}ER_j^{(k_2)}\cdots R_j^{(k_{n})}ER_j^{(k_{n+1})}\|_2\\
&\leq g_j^{n-m}\max(\Vert|R_j|^{1/2}E|R_j|^{1/2}\Vert_\infty,g_j^{-1/2}\Vert |R_j|^{1/2}E P_j\Vert_2,g_j^{-1}\Vert P_j E P_j\Vert_2)^{n},
\end{align*}
which gives the first claim.

Moreover, if $k_a=0$ for some $a\leq n+1$ and if $m\geq 1$, then there is at least on $b\leq n$ such that $\|S_j^{(k_{b})}ES_j^{(k_{b+1})}\|_2=\||R_j|^{1/2}EP_j\|_2$, leading to the second claim.
\end{proof}

The last lemma states that under the condition $\delta_j<1/2$, it is possible to obtain tight bounds for the weighted expression $\||R_j|^{-1/2}\hat P_j\|_2$, by exploiting a contraction property. Later, this term will arise when applying the weighting in the proof of Lemma \ref{LemCoeffBound} to the remainder term from Lemma \ref{LemExplicitRem}.

\begin{lemma}\label{LemER} Suppose that $\delta_j<1/2$. Then we have
\[
\Vert|R_j|^{-1/2}\hat P_j\Vert_2\leq \frac{\Vert|R_j|^{1/2}E P_j\Vert_2}{1-2\delta_j}.
\]
\end{lemma}
\begin{proof}
By \eqref{EqBasicBB} and  Lemma \ref{LemEVSeparation}, we have
\[
\forall k\neq j,\quad \Vert P_k\hat P_j\Vert_2^2=\frac{\Vert P_kE \hat P_j\Vert_2^2}{(\hat \lambda_j-\lambda_k)^2}\leq \frac{1}{(1-\delta_j)^2}\frac{\Vert P_kE \hat P_j\Vert_2^2}{(\lambda_j-\lambda_k)^2}
\]
and thus
\begin{equation}\label{EqNormalizedBasicIneq}
\Vert|R_j|^{-1/2}\hat P_j\Vert_2\leq \frac{\Vert|R_j|^{1/2}E \hat P_j\Vert_2}{1-\delta_j}.
\end{equation}
Applying the triangular inequality, the identities $I=P_j+(I-P_j)$ and $I-P_j=|R_j|^{1/2}|R_j|^{-1/2}$ and the inequality $\delta_j'\leq \delta_j<1/2$, we get
\begin{align}
\Vert|R_j|^{1/2}E \hat P_j\Vert_2&\leq \Vert|R_j|^{1/2}E P_j\hat P_j\Vert_2+\Vert|R_j|^{1/2}E (I-P_i)\hat P_j\Vert_2\nonumber\\
&=\Vert|R_j|^{1/2}E P_j\hat P_j\Vert_2+\Vert|R_j|^{1/2}E |R_j|^{1/2}|R_j|^{-1/2}\hat P_j\Vert_2\nonumber\\
&\leq \Vert|R_j|^{1/2}E P_j\hat P_j\Vert_2+\Vert|R_j|^{1/2}E |R_j|^{1/2}\Vert_\infty\Vert|R_j|^{-1/2}\hat P_j\Vert_2\nonumber\\
&\leq \Vert|R_j|^{1/2}E P_j\Vert_2+\delta_j\Vert|R_j|^{-1/2}\hat P_j\Vert_2\label{EqRecIneq}.
\end{align}
Inserting \eqref{EqRecIneq} into \eqref{EqNormalizedBasicIneq}, we get
\[
\Vert|R_j|^{-1/2}\hat P_j\Vert_2\leq \frac{\Vert|R_j|^{1/2}E P_j\Vert_2}{1-\delta_j}+\frac{\delta_j\Vert|R_j|^{-1/2}\hat P_j\Vert_2}{1-\delta_j},
\]
and the claim follows.
\end{proof}

\subsection{Proof of Theorem \ref{ThmTaylorSP}}
By Lemma \ref{LemExplicitRem} and the triangular inequality, we have
\begin{align}\label{EqExplRemNorm}
\|\hat P_j-\sum_{n=0}^{p-1}P_j^{(n)}\|_2\leq \sum_{k\in\mathbb{Z}}\sum_{\substack{k_1,\dots,k_p\geq 0\\k_1+\dots+k_{p}=p-k}}\|R_j^{(k_1)}E\cdots ER_j^{(k_p)}E\hat R_j^{(k)}\|_2.
\end{align}
We now analyze the right-hand side term by term. For this, let $(k_1,\dots,k_p,k)\in\mathbb{N}_0^p\times\mathbb{Z}$ with $k_1+\dots+k_p=p-k$. We consider separately the cases $k\leq 0$ and $k\geq 1$.  First, for $k\leq 0$, by the identity $I=P_j+(I-P_j)$ and the triangular inequality, we have
\begin{align*}
&\|R_j^{(k_1)}E\cdots ER_j^{(k_p)}E\hat R_j^{(k)}\|_2\\
&=(\hat\lambda_j-\lambda_j)^{-k}\|R_j^{(k_1)}E\cdots ER_j^{(k_p)}E\hat P_j\|_2\\
&\leq (\hat\lambda_j-\lambda_j)^{-k}\|R_j^{(k_1)}E\cdots ER_j^{(k_p)}EP_j\hat P_j\|_2\\
&+(\hat\lambda_j-\lambda_j)^{-k}\|R_j^{(k_1)}E\cdots ER_j^{(k_p)}E(I-P_j)\hat P_j\|_2.
\end{align*}
Thus, by the identity $I-P_j=|R_j|^{1/2}|R_j|^{-1/2}$ and simple properties of the Hilbert-Schmidt norm, we get
\begin{align}
&\|R_j^{(k_1)}E\cdots ER_j^{(k_p)}E\hat R_j^{(k)}\|_2\nonumber\\
&\leq (\hat\lambda_j-\lambda_j)^{-k}\|R_j^{(k_1)}E\cdots ER_j^{(k_p)}EP_j\|_\infty\nonumber\\\
&+(\hat\lambda_j-\lambda_j)^{-k}\|R_j^{(k_1)}E\cdots ER_j^{(k_p)}E|R_j|^{1/2}\|_\infty\||R_j|^{-1/2}\hat P_j\|_2.\label{EqCoeffExcessRisk1}
\end{align}
Now, by Lemma \ref{LemCoeffBound} and Remark \ref{RemCoeffBound}, we have
\begin{align*}
&\|R_j^{(k_1)}E\cdots ER_j^{(k_p)}EP_j\|_\infty\leq g_j^{k}g_j^{-1/2}\Vert |R_j|^{1/2}EP_j\Vert_2{\delta_j'}^{p-1},\\
&\|R_j^{(k_1)}E\cdots ER_j^{(k_p)}E|R_j|^{1/2}\|_\infty\leq g_j^{k-1/2}{\delta_j'}^p.
\end{align*}
Inserting this into \eqref{EqCoeffExcessRisk1} and using Lemmas \ref{LemEVSeparation} and \ref{LemER}, we get for $k\leq 0$,
\begin{align*}
&\|R_j^{(k_1)}E\cdots ER_j^{(k_p)}E\hat R_j^{(k)}\|_2\\
&\leq g_j^{-1/2}\Vert |R_j|^{1/2}EP_j\Vert_2{\delta_j'}^{p-1}\delta_j^{-k}+g_j^{-1/2}\Vert |R_j|^{1/2}EP_j\Vert_2{\delta_j'}^{p}\frac{\delta_j^{-k}}{1-2\delta_j}\\
&\leq g_j^{-1/2}\Vert |R_j|^{1/2}EP_j\Vert_2{\delta_j'}^{p-1}\frac{\delta_j^{-k}}{1-2\delta_j}.
\end{align*}
From this and Lemma \ref{LemNumberCoeff}, we conclude that
\begin{align}
&\sum_{k\leq 0}\sum_{\substack{k_1,\dots,k_p\geq 0\\k_1+\dots+k_{p}=p-k}}\|R_j^{(k_1)}E\cdots ER_j^{(k_p)}E\hat R_j^{(k)}\|_2\nonumber\\
&\leq g_j^{-1/2}\Vert |R_j|^{1/2}EP_j\Vert_2\frac{{\delta_j'}^{p-1}}{1-2\delta_j}\sum_{k\leq 0}2^{2p-1-k}\delta_j^{-k}\nonumber\\
&=2g_j^{-1/2}\Vert |R_j|^{1/2}EP_j\Vert_2\frac{(4{\delta_j'})^{p-1}}{(1-2\delta_j)^2}.\label{EqRemainderCoeffBound1}
\end{align}
Next, consider the case $k\geq 1$. Then we have
\begin{align}
&\|R_j^{(k_1)}E\cdots ER_j^{(k_p)}E\hat R_j^{(k)}\|_2\nonumber\\
&=\|R_j^{(k_1)}E\cdots ER_j^{(k_p)}E\hat R_j^{k}\|_2\nonumber\\
&\leq \|R_j^{(k_1)}E\cdots ER_j^{(k_p)}EP_j\|_2\|P_j\hat R_j^{k}\|_2\nonumber\\
&+\|R_j^{(k_1)}E\cdots ER_j^{(k_p)}E|R_j|^{1/2}\|_2\||R_j|^{-1/2}\hat R_j^{k}\|_\infty\label{EqCoeffExcessRisk2}.
\end{align}
By  Lemma \ref{LemCoeffBound} and Remark \ref{RemCoeffBound}, using the fact that $k\geq 1$ implies $k_j=0$ for some $j\leq p$, we have
\begin{align*}
&\|R_j^{(k_1)}E\cdots ER_j^{(k_p)}EP_j\|_2\leq g_j^{k}{\delta_j'}^{p},\\
&\|R_j^{(k_1)}E\cdots ER_j^{(k_p)}E|R_j|^{1/2}\|_2\leq g_j^{k-1}\Vert P_jE|R_j|^{1/2}\Vert_2{\delta_j'}^{p-1}.
\end{align*}
By Lemma \ref{LemEVSeparation}, we have
\begin{align*}
\|P_j\hat R_j^{k}\|_2&=\sqrt{\sum_{k\neq j}\frac{1}{(\hat\lambda_k-\lambda_j)^{2k}}\|P_j\hat P_k\|_2^2}\\
&\leq (1-\delta_j)^{-k}g_j^{-k}\|P_j(I-\hat P_j)\|_2\\
&=(1-\delta_j)^{-k}g_j^{-k}\|(I-P_j)\hat P_j\|_2\leq (1-\delta_j)^{-k}g_j^{-k-1/2}\||R_j|^{-1/2}\hat P_j\|_2
\end{align*}
and thus by Lemma \ref{LemER}
\begin{equation}\label{EqCoeffBoundRem1}
\|P_j\hat R_j^{k}\|_2\leq g_j^{-k-1/2}\||R_j|^{1/2}E P_j\|_2\frac{(1-\delta_j)^{-k}}{1-2\delta_j}
\end{equation}
Similarly we have
\begin{align}
\||R_j|^{-1/2}\hat R_j^{k}\|_\infty\leq (1-\delta_j)^{-k}g_j^{-k+1/2}.\label{EqCoeffBoundRem2}
\end{align}
Inserting these inequalities into \eqref{EqCoeffExcessRisk2}, we get for $k\geq 1$,
\begin{equation*}
\|R_j^{(k_1)}E\cdots ER_j^{(k_p)}E\hat R_j^{k}\|_2\leq g_j^{-1/2}\Vert P_jE|R_j|^{1/2}\Vert_2{\delta_j'}^{p-1}\frac{(1-\delta_j)^{-k}}{1-2\delta_j}.
\end{equation*}
From this and Lemma \ref{LemNumberCoeff}, we conclude that
\begin{align}
&\sum_{k=1}^p\sum_{\substack{k_1,\dots,k_p\geq 0\\k_1+\dots+k_{p}=p-k}}\|R_j^{(k_1)}E\cdots ER_j^{(k_p)}E\hat R_j^{(k)}\|_2\nonumber\\
&\leq g_j^{-1/2}\Vert P_jE|R_j|^{1/2}\Vert_2\frac{{\delta_j'}^{p-1}}{1-2\delta_j}\sum_{k=1}^\infty2^{2p-1-k}(1-\delta_j)^{-k}\nonumber\\
&=2g_j^{-1/2}\Vert P_jE|R_j|^{1/2}\Vert_2\frac{(4{\delta_j'})^{p-1}}{(1-2\delta_j)^2}.\label{EqRemainderCoeffBound2}
\end{align}
Inserting \eqref{EqRemainderCoeffBound1} and \eqref{EqRemainderCoeffBound2} into \eqref{EqExplRemNorm} completes the proof.\qed

\subsection{Proof of Theorem \ref{ThmTaylorEV}}
We have
\begin{equation*}
\hat\lambda_j-\lambda_j=\operatorname{tr}(\hat P_jE)+ \operatorname{tr}(\hat P_jR_j^{-1})
\end{equation*}
and thus by  Lemma \ref{LemExplicitRem} and the triangular inequality
\begin{align}
|\hat \lambda_j-\sum_{n=0}^{p-1}\lambda_j^{(n)}|&\leq \sum_{k\in\mathbb{Z}}\sum_{\substack{k_1,\dots,k_{p-1}\geq 0\\k_1+\dots+k_{p-1}=p-1-k}}|\operatorname{tr}(ER_j^{(k_1)}E\cdots ER_j^{(k_{p-1})}E\hat R_j^{(k)})|\nonumber\\
&\quad+\sum_{k\in\mathbb{Z}}\sum_{\substack{k_1\geq 1,k_2,\dots,k_{p}\geq 0\\k_1+\dots+k_{p}=p-k}}|\operatorname{tr}(R_j^{k_1-1}E\cdots ER_j^{(k_{p})}E\hat R_j^{(k)})|\label{EqExplRemNormEV}.
\end{align}
We first consider the first term on the right-hand side of \eqref{EqExplRemNormEV}. For $k\leq 0$, we have
\begin{align*}
&|\operatorname{tr}(ER_j^{(k_1)}E\cdots ER_j^{(k_{p-1})}E\hat R_j^{(k)})|\\
&=|\hat\lambda_j-\lambda_j|^{-k}|\operatorname{tr}(\hat P_jER_j^{(k_1)}E\cdots ER_j^{(k_{p-1})}E\hat P_j)|\\
&\leq |\hat\lambda_j-\lambda_j|^{-k}\|\hat P_jER_j^{(k_1)}E\cdots ER_j^{(k_{p-1})}E\hat P_j\|_2.
\end{align*}
Inserting $I=P_j+|R_j|^{1/2}|R_j|^{-1/2}$ twice, we have
\begin{align*}
&\|\hat P_jER_j^{(k_1)}E\cdots ER_j^{(k_{p-1})}E\hat P_j\|_2\\
&\leq \| P_jER_j^{(k_1)}E\cdots ER_j^{(k_{p-1})}E P_j\|_2\\
&+\||R_j|^{-1/2}\hat P_j\|_2\| P_jER_j^{(k_1)}E\cdots ER_j^{(k_{p-1})}E|R_j|^{1/2}\|_2\\
&+\||R_j|^{-1/2}\hat P_j\|_2\||R_j|^{1/2}ER_j^{(k_1)}E\cdots ER_j^{(k_{p-1})}EP_j\|_2\\
&+\||R_j|^{-1/2}\hat P_j\|_2^2\||R_j|^{1/2}ER_j^{(k_1)}E\cdots ER_j^{(k_{p-1})}E|R_j|^{1/2}\|_2.
\end{align*}
By Lemmas \ref{LemEVSeparation}, \ref{LemCoeffBound}, and \ref{LemER}, we get for $k\leq 0$,
\begin{align}
&|\operatorname{tr}(R_j^{(k_1)}E\cdots ER_j^{(k_{p-1})}E\hat R_j^{(k)}E)|\nonumber\\
&\leq \Vert P_jE|R_j|^{1/2}\Vert_2^2\Big({\delta_j'}^{p-2}\delta_j^{-k}+\frac{2{\delta_j'}^{p-1}\delta_j^{-k}}{1-2\delta_j}+\frac{{\delta_j'}^{p}\delta_j^{-k}}{(1-2\delta_j)^2}\Big)\nonumber\\
&\leq \Vert P_jE|R_j|^{1/2}\Vert_2^2{\delta_j'}^{p-2}\frac{\delta_j^{-k}}{(1-2\delta_j)^2}.\label{EqCoeffBoundEV1}
\end{align}
Similarly, for $k\geq 1$, using the fact that $k_a=0$ for some $a\leq p-1$, we have
\begin{align}
&|\operatorname{tr}(R_j^{(k_1)}E\cdots ER_j^{(k_{p-1})}E\hat R_j^{(k)}E)|\nonumber\\
&=\operatorname{tr}(P_jER_j^{(k_{a+1})}E\cdots ER_j^{(k_{p-1})}E\hat R_j^kER_j^{(k_1)}E\cdots ER_j^{(k_{a-1})}EP_j)\nonumber\\
&\leq \||P_j\hat R_j^{(k)}P_j\|_\infty\|P_jER_j^{(k_{a+1})}\cdots R_j^{(k_{p-1})}EP_j\|_2\|P_jER_j^{(k_1)}\cdots R_j^{(k_{a-1})}EP_j\|_2\nonumber\\
&+\|P_j\hat R_j^{(k)}|R_j|^{-1/2}\|_\infty\|P_jER_j^{(k_{a+1})}\cdots R_j^{(k_{p-1})}EP_j\|_2\nonumber\\
&\ \ \cdot\||R_j|^{1/2}ER_j^{(k_1)}\cdots R_j^{(k_{a-1})}EP_j\|_2\nonumber\\
&+\||R_j|^{-1/2}\hat R_j^{(k)}P_j\|_\infty\|P_jER_j^{(k_{a+1})}\cdots R_j^{(k_{p-1})}E|R_j|^{1/2}\|_2\nonumber\\
&\ \ \cdot\|P_jER_j^{(k_1)}\cdots R_j^{(k_{a-1})}EP_j\|_2\nonumber\\
&+\||R_j|^{-1/2}\hat R_j^{(k)}|R_j|^{-1/2}\|_\infty\|P_jER_j^{(k_{a+1})}\cdots R_j^{(k_{p-1})}E|R_j|^{1/2}\|_2\nonumber\\
&\ \ \cdot\||R_j|^{1/2}ER_j^{(k_1)}\cdots R_j^{(k_{a-1})}EP_j\|_2\label{EqCoeffBoundEVk1}.
\end{align}
By \eqref{EqCoeffBoundEVPart1} and \eqref{EqCoeffBoundEVPart2}, we have
\begin{align}
&\||P_j\hat R_j^{(k)}P_j\|_\infty \leq g_j^{-k}g_j^{-1}\Vert P_jE|R_j|^{1/2}\Vert_2^2\frac{(1-\delta_j)^{-k}}{(1-2\delta_j)^2},\nonumber\\
&\|P_j\hat R_j^{(k)}|R_j|^{-1/2}\|_\infty\leq g_j^{-k}\||R_j|^{1/2}E P_j\|_2\frac{(1-\delta_j)^{-k}}{1-2\delta_j},\nonumber\\
&\||R_j|^{-1/2}\hat R_j^{(k)}|R_j|^{-1/2}\|_\infty\leq (1-\delta_j)^{-k}g_j^{-k+1}\label{EqCoeffBoundEVRem}.
\end{align}
Inserting these inequalities, Lemma \ref{LemCoeffBound}, and Remark \ref{RemCoeffBound} (using the fact that $k_a=0$ for some $a\leq p-1$), we get for $k\geq 1$,
\begin{align}
&|\operatorname{tr}(ER_j^{(k_1)}E\cdots ER_j^{(k_{p-1})}E\hat R_j^{(k)})|\nonumber\\
&\leq \Vert P_jE|R_j|^{1/2}\Vert_2^2\Big({\delta_j'}^{p-2}+\frac{2{\delta_j'}^{p-1}}{1-2\delta_j}+\frac{{\delta_j'}^{p}}{(1-2\delta_j)^2}\Big)(1-\delta_j)^{-k}\nonumber\\
&\leq \Vert P_jE|R_j|^{1/2}\Vert_2^2{\delta_j'}^{p-2}\frac{(1-\delta_j)^{-k}}{(1-2\delta_j)^2}.\label{EqCoeffBoundEV2}
\end{align}
Using \eqref{EqCoeffBoundEV1}, \eqref{EqCoeffBoundEV2}, and Lemma \ref{LemNumberCoeff}, we get
\begin{align}
&\sum_{k\in\mathbb{Z}}\sum_{\substack{k_1,\dots,k_{p-1}\geq 0\\k_1+\dots+k_{p-1}=p-1-k}}\operatorname{tr}(ER_j^{(k_1)}E\cdots ER_j^{(k_{p-1})}E\hat R_j^{(k)})\nonumber\\
&\leq 4\Vert P_jE|R_j|^{1/2}\Vert_2^2\frac{(4\delta_j')^{p-2}}{(1-2\delta_j)^3}\label{EqCoeffBoundEVPart1}.
\end{align}
Next, we consider the second term on the right-hand side of \eqref{EqExplRemNormEV}.
For $k\leq 0$, we have
\begin{align*}
&|\operatorname{tr}(R_j^{k_1-1}E\cdots ER_j^{(k_{p})}E\hat R_j^{(k)})|\\
&\leq |\hat\lambda_j-\lambda_j|^{-k}\|\hat P_jR_j^{k_1-1}E\cdots ER_j^{(k_{p})}E\hat P_j\|_2\\
&\leq |\hat\lambda_j-\lambda_j|^{-k}\|\hat P_j|R_j|^{-1/2}\|_2\||R_j|^{k_1-1/2}E\cdots ER_j^{(k_{p})}EP_j\|_2\\
&+|\hat\lambda_j-\lambda_j|^{-k}\|\hat P_j|R_j|^{-1/2}\|_2^2\||R_j|^{k_1-1/2}E\cdots ER_j^{(k_{p})}E|R_j|^{1/2}\|_2,
\end{align*}
and thus, by Lemmas \ref{LemEVSeparation}, \ref{LemCoeffBound}, and \ref{LemER}, we get
\begin{align}
&|\operatorname{tr}(R_j^{k_1-1}E\cdots ER_j^{(k_{p})}E\hat R_j^{(k)})|\nonumber\\
&\leq \Vert P_jE|R_j|^{1/2}\Vert_2^2\Big(\frac{{\delta_j'}^{p-1}\delta_j^{-k}}{1-2\delta_j}+\frac{{\delta_j'}^{p}\delta_j^{-k}}{(1-2\delta_j)^2}\Big)\nonumber\\
&\leq \Vert P_jE|R_j|^{1/2}{\delta_j'}^{p-2}\Vert_2^2\frac{\delta_j^{-k}}{(1-2\delta_j)^2}\label{EqCoeffBoundEV3}.
\end{align}
On the other hand, for $k\geq 1$, we use the fact that $k_a=0$ for some $a\leq p-1$ to obtain
\begin{align*}
&|\operatorname{tr}(R_j^{k_1-1}E\cdots ER_j^{(k_{p})}E\hat R_j^{k})|\\
&\leq \|P_jER_j^{(k_{a+1})}E\cdots ER_j^{(k_{p})}E\hat R_j^kR_j^{k_1-1}E\cdots ER_j^{(k_{a-1})}EP_j)\|_2\\
&\leq \||R_j|^{-1/2}\hat R_j^{(k)}|R_j|^{-1/2}\|_\infty\|P_jER_j^{(k_{a+1})}E\cdots ER_j^{(k_{p})}E|R_j|^{1/2}\|_2\\
&\ \ \cdot\||R_j|^{k_1-1/2}E\cdots ER_j^{(k_{a-1})}EP_j\|_2\\
&+ \|P_j\hat R_j^{(k)}|R_j|^{-1/2}\|_\infty\|P_jER_j^{(k_{a+1})}E\cdots ER_j^{(k_{p})}EP_j\|_2\\
&\ \ \cdot\||R_j|^{k_1-1/2}E\cdots ER_j^{(k_{a-1})}EP_j\|_2.
\end{align*}
Thus, by Lemmas \ref{LemEVSeparation} and \ref{LemCoeffBound} and \eqref{EqCoeffBoundEVRem}, we get
\begin{align}
&|\operatorname{tr}(R_j^{k_1-1}E\cdots ER_j^{(k_{p})}E\hat R_j^{k})|\nonumber\\
&\leq \Vert P_jE|R_j|^{1/2}\Vert_2^2\Big({\delta_j'}^{p-2}+\frac{{\delta_j'}^{p-1}}{1-2\delta_j}\Big)(1-\delta_j)^{-k}\nonumber\\
&\leq \Vert P_jE|R_j|^{1/2}\Vert_2^2{\delta_j'}^{p-2}\frac{(1-\delta_j)^{-k}}{1-2\delta_j}.\label{EqCoeffBoundEV4}
\end{align}
Using \eqref{EqCoeffBoundEV3}, \eqref{EqCoeffBoundEV4}, and Lemma \ref{LemNumberCoeff}, we get
\begin{align}
&\sum_{k\in\mathbb{Z}}\sum_{\substack{k_1\geq 1,k_2,\dots,k_{p}\geq 0\\k_1+\dots+k_{p}=p-k}}|\operatorname{tr}(R_j^{k_1-1}E\cdots ER_j^{(k_{p})}E\hat R_j^{(k)})|\nonumber\\
&\leq 8\Vert P_jE|R_j|^{1/2}\Vert_2^2\frac{(4\delta_j')^{p-2}}{(1-2\delta_j)^3}.\label{EqCoeffBoundEVPart2}
\end{align}
Inserting \eqref{EqCoeffBoundEVPart1} and \eqref{EqCoeffBoundEVPart2} into \eqref{EqExplRemNormEV} completes the proof.
\qed

\subsection{Proofs for the consequences}
\begin{proof}[Proof of Corollary \ref{CorTaylorExpPS}]
By Lemmas \ref{LemNumberCoeff} and \ref{LemCoeffBound}, we have $\|P_j^{(n)}\|_2\leq g_j^{-1/2}\Vert P_jE|R_j|^{1/2}\Vert_24^n{\delta_j'}^{n-1}$ and thus by Corollary \ref{CorPertSeries} and the triangular inequality,
\begin{align*}
\|\hat P_j-\sum_{n=0}^{p-1}P_j^{(n)}\|_2\leq \sum_{n\geq p}\|P_j^{(n)}\|_2\leq 4g_j^{-1/2}\Vert P_jE|R_j|^{1/2}\Vert_2\frac{(4\delta_j')^{p-1}}{1-4\delta_j'}.
\end{align*}
To obtain the second claim note that by Lemmas \ref{LemNumberCoeff} and \ref{LemCoeffBound}, we have
\begin{align*}
|\lambda_j^{(n)}|&\leq |\operatorname{tr}(P_j^{(n-1)}E)|+|\operatorname{tr}(P_j^{(n)}R_j^{-1})|\\
&\leq (4^{n-1}+4^{n-1}){\delta_j'}^{n-2}\Vert P_jE|R_j|^{1/2}\Vert_2^2,
\end{align*}
and the second claim follows similarly as above by Corollary \ref{CorPertSeries} and the triangular inequality.
\end{proof}

\begin{proof}[Proof of Corollary \ref{CorImprovedBound}]
 Let us first show that for every $(k_1,\dots,k_{n+1})\in\mathbb{N}_{0}^{n+1}$ with $k_1+\dots+k_{n+1}=n$, there is a $m\leq n$ such that
\begin{equation}\label{EqImprovedCoeffBound}
\|R_j^{(k_1)}ER_j^{(k_2)}\cdots R_j^{(k_{n})}ER_j^{(k_{n+1})}\|_2\leq \|(R_jE)^m P_j\|_2{\delta_j'}^{n-m}.
\end{equation}
To show this, let $a_1,\dots,a_r$, be the indices such that $k_{a_1}=\dots=k_{a_r}=0$, leading to
\begin{align*}
&\|R_j^{(k_1)}ER_j^{(k_2)}\cdots R_j^{(k_{p})}ER_j^{(k_{n+1})}\|_2\\
&=\|R_j^{(k_1)}ER_j^{(k_2)}\cdots R_j^{(k_{a_1-1})}EP_j\|_2\|P_jER_j^{(k_{a_1+1})}\cdots R_j^{(k_{a_2-1})}EP_j\|_2\\
&\quad\cdot\dots\cdot\|P_jER_j^{(k_{a_{r}+1})}\cdots R_j^{(k_{n+1})}\|_2.
\end{align*}
By assumption, we either have $r=1$ and $k_a=1$ for all $a\neq a_1$, or there is either a term in the product which is of the form $\|P_j(ER_j)^{m-1}ER_j^{k}\cdots\|_2\leq \|P_j(ER_j)^m\|_2\|R_j^{k-1}\cdots\|_2$ with $m\leq n-1$ and $k\geq 2$. Combining this observation with Lemma \ref{LemCoeffBound}, we get \eqref{EqImprovedCoeffBound}. Moreover, the number of $n+1$-tuple such that $k_1+\dots+k_{n+1}=n$ and such that a term $\|P_j(ER_j)^{m-1}ER_j^{k}\cdots\|_2$ with $m\leq n-1$ and $k\geq 2$ exists in the above product is bounded by $2(n-m+1)4^{n-m}$, as can be seen by using Lemma \ref{LemNumberCoeff}. Applying Corollary \ref{CorPertSeries} and the triangular inequality, we arrive at
\begin{align*}
\|\hat P_j-P_j\|_2&\leq \sum_{m=1}^\infty\sum_{n=m}^\infty\|P_j(ER_j)^m\|_2(2n-2m+4)(4\delta_j')^{n-m}\\
&\leq C\sum_{m=1}^\infty\|P_j(ER_j)^m\|_2,
\end{align*}
and the claim follows.
\end{proof}

\begin{proof}[Proof of Corollary \ref{CorImprovedBound2}]
Separate the cases $\delta_j'< 1/8$ and $\delta_j'\geq  1/8$. In the former case, the claim follows from Corollary \ref{CorImprovedBound} and Lemma \ref{LemCoeffBound}. In the latter case, we use $\|\hat P_j-P_j\|_2\leq\sqrt{2}\leq \sqrt{2}(8\delta_j')^p$, and the claim follows.
\end{proof}

\begin{proof}[Proof of Corollary \ref{CorImprovedBound3}]
By \eqref{EqBasicBB}, we have for every $k\neq j$,
\[
P_k\hat P_j=\frac{1}{\lambda_j-\lambda_k}P_kE\hat P_j+\frac{\lambda_j-\hat\lambda_j}{\lambda_j-\lambda_k}P_k\hat P_j.
\]
Summing over $k\neq j$ yields
\begin{align}
(I-P_j)\hat P_j
&=R_jE\hat P_j+(\lambda_j-\hat\lambda_j)R_j\hat P_j\nonumber\\
&=R_jEP_j\hat P_j+R_jE(I-P_j)\hat P_j+(\lambda_j-\hat\lambda_j)R_j(I-P_j)\hat P_j\label{EqBasicBBVar}.
\end{align}
Applying \eqref{EqBasicBBVar} $(p-1)$-times, we get
\begin{align}
(I-P_j)\hat P_j
&=\sum_{m=1}^{p-1}(R_jE)^mP_j\hat P_j+(R_jE)^{p-1}(I-P_j)\hat P_j\nonumber\\
&+(\lambda_j-\hat\lambda_j)\sum_{m=1}^{p-1}(R_jE)^{m-1}R_j(I-P_j)\hat P_j.\label{EqPartialExp}
\end{align}
By Lemmas \ref{LemEVSeparation} and \ref{LemCoeffBound}, we have
\begin{align}\label{EqHSPartTaylorExp}
&|\lambda_j-\hat\lambda_j|\Vert\sum_{m=1}^{p-1}(R_jE)^{m-1}R_j\Vert_\infty\leq g_j\delta_j\sum_{m=1}^{p-1}g_j^{-1}\delta_j^{m-1}\leq \frac{\delta_j}{1-\delta_j}<1.
\end{align}
Taking the Hilbert-Schmidt norm in \eqref{EqPartialExp} and inserting \eqref{EqHSPartTaylorExp}, we get
\begin{align}\label{EqRemEst}
\Vert(I-P_j)\hat P_j\Vert_2&\leq \frac{1}{1-2\delta_j}\Big(\Vert\sum_{m=1}^{p-1}(R_jE)^mP_j\Vert_2+\Vert(R_jE)^{p-1}(I-P_j)\hat P_j\Vert_2\Big).
\end{align}
Applying the identity $I-P_j=|R_j|^{1/2}|R_j|^{-1/2}$ and Lemma \ref{LemER}, we get
\[
\Vert(R_jE)^{p-1}(I-P_j)\hat P_j\Vert_2\leq \Vert(R_jE)^{p-1}|R_j|^{1/2}\Vert_\infty\Vert|R_j|^{-1/2}\hat P_j\Vert_2\leq \frac{1}{1-2\delta_j}\delta_j'^{p}.
\]
Inserting this into \eqref{EqRemEst} we obtain Corollary \ref{CorImprovedBound3} under the Condition $\delta_j<1/4$. The latter condition can be dropped by proceeding similarly as in the proof of Corollary \ref{CorImprovedBound2}.
\end{proof}
\section{Applications}\label{SecAppli}
In this section, we apply our results to the the empirical covariance operator, a central object in high-dimensional probability and statistics. Additionally, we show how to obtain similar conclusions in the case of kernel operators and kernel Gram matrices.
\subsection{The empirical covariance operator}
Let $X$ be a random variable taking values in $\mathcal{H}$. We suppose that $X$ is centered and strongly square-integrable, meaning that $\mathbb{E} X =0$ and $\mathbb{E}\|X\|^2<\infty$. Let $\Sigma =\mathbb{E} X\otimes X$ be the covariance operator of $X$, which is a positive, self-adjoint trace class operator, see e.g. \cite[Theorem 7.2.5]{HE15}. Let $X_1,\dots,X_n$ be independent copies of $X$ and let
\begin{equation*}
\hat{\Sigma}=\frac{1}{n}\sum_{i=1}^nX_i\otimes X_i
\end{equation*}
be the empirical covariance operator.

\begin{assumption}\label{SubGauss}
Suppose that $X$ is sub-Gaussian, meaning that there is a constant $\SGN$ with
\[
\forall u\in\mathcal{H},\qquad \sup_{q\ge 1}q^{-1/2}\mathbb{E}^{1/q} |\langle X,u\rangle|^{q} \le \SGN\mathbb{E}^{1/2}\langle X,u\rangle^2.
\]
\end{assumption}
For $j\geq 1$, consider  $X'=(|R_j|^{1/2}+g_j^{-1/2}P_j)X$, $X_i'=(|R_j|^{1/2}+g_j^{-1/2}P_j)X_i$ which again satisfy Assumption \ref{SubGauss} (with the same constant $\SGN$) and lead to the covariance and the sample covariance
\begin{align*}\label{EqSigma'}
\Sigma'&=(|R_j|^{1/2}+g_j^{-1/2}P_j)\Sigma (|R_j|^{1/2}+g_j^{-1/2}P_j),\\
 \hat\Sigma'&=(|R_j|^{1/2}+g_j^{-1/2}P_j)\hat\Sigma (|R_j|^{1/2}+g_j^{-1/2}P_j).
 \end{align*}
Thus we have
\[
\delta_j=\Vert(|R_j|^{1/2}+g_j^{-1/2}P_j)E(|R_j|^{1/2}+g_j^{-1/2}P_j)\Vert_\infty=\|\hat\Sigma'-\Sigma'\|_\infty.
\]
This observation allows us to transfer the results from \cite{KL14} to $\delta_j$. First, we state a high probability result for $\delta_j$:
\begin{lemma}\label{KoltLounConc} Under Assumption \ref{SubGauss}, there are constants $c_1,c_2>0$ depending only on $\SGN$ such that for every $j\geq 1$ satisfying
\begin{equation}\label{EqRGCond}
\frac{\lambda_j}{g_j}\Big(\sum_{k\neq j}\frac{\lambda_k}{|\lambda_k-\lambda_j|}+\frac{\lambda_j}{g_j}\Big)\leq c_1n,
\end{equation}
we have
\begin{equation*}
\mathbb{P}\big(\delta_j> 1/4\big)\leq e^{-c_2ng_j^2/\lambda_j^2}.
\end{equation*}
\end{lemma}
Second, in order to bound the remainder term in Theorems \ref{ThmTaylorSP} and \ref{ThmTaylorEV}, we will apply the following moment bound:
\begin{lemma}\label{KoltLounExp} Suppose that Assumption \ref{SubGauss} holds. Then, for every  $p\geq 1$, there is a constant $C_1>1$ depending only on $\SGN$ and $p$ such that for every $j\geq 1$,
\begin{align*}
& \mathbb{E}^{1/p}\delta_j^p\leq C_1\max\bigg(\sqrt{\frac{1}{n}\frac{\lambda_j}{g_j}\Big(\sum_{k\neq j}\frac{\lambda_k}{|\lambda_k-\lambda_j|}+\frac{\lambda_j}{g_j}\Big)}, \frac{1}{n}\Big(\sum_{k\neq j}\frac{\lambda_k}{|\lambda_k-\lambda_j|}+\frac{\lambda_j}{g_j}\Big)\bigg).
\end{align*}
If additionally \eqref{EqRGCond} holds, then we have
\begin{align*}
\mathbb{E}^{1/p}\delta_j^p\leq C_1\sqrt{\frac{1}{n}\frac{\lambda_j}{g_j}\Big(\sum_{k\neq j}
\frac{\lambda_k}{|\lambda_k-\lambda_j|}+\frac{\lambda_j}{g_j}\Big)}.
\end{align*}
\end{lemma}

\subsection{Empirical eigenvalues}
In this section, we will apply Corollary \ref{PertBoundEV} and Theorem \ref{ThmTaylorEV} to the eigenvalues of the empirical covariance operator.
\begin{thm}\label{CorEEVUpperBound}
If Assumption \ref{SubGauss} holds, then there are constants $c_1,c_2,C_1>0$ depending only on $\SGN$ such that for all $j\geq 1$ satisfying \eqref{EqRGCond},
\begin{align}
\mathbb{E}^{1/2}(\hat\lambda_j-\lambda_j)^2
&\leq C_1\frac{\lambda_j}{\sqrt{n}}\Big(1+ \frac{1}{\sqrt{n}}\sum_{k\neq j}\frac{\lambda_k}{|\lambda_k-\lambda_j|}\Big)\nonumber\\
&\quad+C_1e^{-c_2ng_j^2/\lambda_j^2}\Big(\sqrt{\frac{\lambda_1\operatorname{tr}(\Sigma)}{n}}+ \frac{\operatorname{tr}(\Sigma)}{n}\Big).\label{EqEEVUpperBound}
\end{align}
Moreover, if $X$ is Gaussian, then there are absolute constants $c_1,C_1>0$ such that for all $j\geq 1$ satisfying \eqref{EqRGCond},
\begin{align}
\mathbb{E}^{1/2}(\hat\lambda_j-\lambda_j)^2
&\geq C_1^{-1}\frac{\lambda_j}{\sqrt{n}}\Big(1+ \frac{1}{\sqrt{n}}\Big|\sum_{k\neq j}\frac{\lambda_k}{\lambda_k-\lambda_j}\Big|\Big)\nonumber\\
&\quad-C_1\frac{\lambda_j}{n}\sum_{k\neq j}\frac{\lambda_k}{|\lambda_k-\lambda_j|}\sqrt{\frac{1}{n}\frac{\lambda_j}{g_j}\Big(\sum_{k\neq j}\frac{\lambda_k}{|\lambda_k-\lambda_j|}+\frac{\lambda_j}{g_j}\Big)}.\label{EqEEVLowerBound}
\end{align}
\end{thm}
\begin{remark}
Inequalities \eqref{EqEEVUpperBound} and \eqref{EqEEVLowerBound} provide matching upper and lower bounds, provided that the absolute value of $\sum_{k\neq j}\lambda_k/(\lambda_k-\lambda_j)$ is comparable to the sum of the absolute values $\sum_{k\neq j}\lambda_k/|\lambda_k-\lambda_j|$, as can be seen by inserting \eqref{EqRGCond} into the remainder term in \eqref{EqEEVLowerBound}. For instance, in the case $j=1$, all terms in the sum $\sum_{k>1}\lambda_k/(\lambda_1-\lambda_k)$ are positive, meaning that we indeed obtain matching upper and lower bounds (up to an exponentially small remainder term).
\end{remark}
\begin{remark}
Theorem \ref{CorEEVUpperBound} reveals a phase transition. First, if
\begin{equation}\label{EqRelRankCond}
\sum_{k\neq j}\frac{\lambda_k}{|\lambda_k-\lambda_j|}+\frac{\lambda_j}{g_j}\leq \sqrt{c_1n}
\end{equation}
holds (which implies \eqref{EqRGCond}), then the $L^2$-norm of the linear perturbation term dominates the bound. On the other hand, if \eqref{EqRGCond} holds, but \eqref{EqRelRankCond} does not hold, then the second order perturbation term dominates the bound; see Section \ref{SecRRC} for more discussion.
\end{remark}
\begin{remark}\label{RemEVSep}
The second order perturbation term in the upper bound \eqref{EqEEVUpperBound} can also be written as
\[
C_1g_j\cdot \frac{\lambda_j}{g_j}\sum_{k\neq j}\frac{\lambda_k}{|\lambda_k-\lambda_j|}.
\]
Hence, as long as \eqref{EqRGCond} is satisfied with $c_1$ small enough (e.g. such that the above term is bounded by $g_j/2$), we have an eigenvalue separation property. We conjecture that a reverse inequality holds if \eqref{EqRGCond} does not hold, in which case we could not even cluster empirical and population eigenvalues (and eigenprojections) appropriately.
\end{remark}

\begin{proof}[Proof of Theorem \ref{CorEEVUpperBound}]
By Minkowski's inequality, we have
\[
\mathbb{E}^{1/2}(\hat\lambda_j-\lambda_j)^2\leq \mathbb{E}^{1/2}\mathbf{1}(\delta_j\leq 1/4)(\hat\lambda_j-\lambda_j)^2+\mathbb{E}^{1/2}\mathbf{1}(\delta_j> 1/4)(\hat\lambda_j-\lambda_j)^2.
\]
Applying Corollary \ref{PertBoundEV} and Minkowski's inequality to the first term, and the Weyl
bound $|\hat\lambda_j-\lambda_j|\leq \|E\|_\infty$ and the Cauchy-Schwarz inequality to the second
one, we obtain that
\begin{align}
\mathbb{E}^{1/2}(\hat\lambda_j-\lambda_j)^2&\leq \mathbb{E}^{1/2}\|P_jEP_j\|_2^2+C\mathbb{E}^{1/2}\Vert P_jE|R_j|^{1/2}\Vert_2^4\nonumber\\
&+(\mathbb{E}\|E\|^4_\infty)^{1/4}(\mathbb{P}(\delta_j>1/4))^{1/4}.\label{EqEEVUpperBoundPr}
\end{align}
By Assumption \ref{SubGauss} and Minkowski's inequality, we have
\[
\mathbb{E}^{1/2}\|P_jEP_j\|_2^2\leq C\frac{\lambda_j}{\sqrt{n}},\qquad\mathbb{E}^{1/2}\Vert P_jE|R_j|^{1/2}\Vert_2^4\leq C\frac{\lambda_j}{n}\sum_{k\neq j}\frac{\lambda_k}{|\lambda_j-\lambda_k|}
\]
with a constant $C>0$ depending only on $\SGN$. Inserting these inequalities, Lemma \ref{KoltLounExp} and [13, Corollary 2] into \eqref{EqEEVUpperBoundPr}, \eqref{EqEEVUpperBound} follows.

Next, by Minkowski's inequality, we have
\begin{align}
&\mathbb{E}^{1/2}(\hat\lambda_j-\lambda_j)^2\nonumber\\
&\geq (\mathbb{E}(\operatorname{tr}(P_jEP_j)-\operatorname{tr}(P_jER_jEP_j))^2)^{1/2}\nonumber\\
&-(\mathbb{E}\mathbf{1}(\delta_j\leq 1/4)(\hat\lambda_j-\lambda_j-\operatorname{tr}(P_jEP_j)+\operatorname{tr}(P_jER_jEP_j))^2)^{1/2}\nonumber\\
&-(\mathbb{E}\mathbf{1}(\delta_j>  1/4)(\operatorname{tr}(P_jEP_j)-\operatorname{tr}(P_jER_jEP_j))^2)^{1/2}=:I_1-I_2-I_3.\label{EqLowBound1}
\end{align}
By a simple moment computation for Gaussian chaos, we have
\begin{align*}
I_1^2&=\frac{\lambda_j^2}{n}\Big(2+\frac{n+2}{n^2}\Big(\sum_{k\neq j}\frac{\lambda_k}{\lambda_k-\lambda_j}\Big)^2+\frac{6}{n^2}\sum_{k\neq j}\frac{\lambda_k^2}{(\lambda_k-\lambda_j)^2}+\frac{2}{n}\sum_{k\neq j}\frac{\lambda_k}{\lambda_k-\lambda_j}\Big).
\end{align*}
Using the inequality $x^2-2x\geq x^2/3-3/2$, $x\geq 0$, we get
\begin{align*}
I_1^2\geq \frac{\lambda_j^2}{n}\Big(2-\frac{3}{2n}+\frac{1}{3n}\Big(\sum_{k\neq j}\frac{\lambda_k}{\lambda_k-\lambda_j}\Big)^2\Big),
\end{align*}
and thus, using $\sqrt{1+x}\geq (1+\sqrt{x})/2$, $x\geq 0$,
\begin{align*}
I_1\geq \frac{1}{6}\frac{\lambda_j}{\sqrt{n}}\Big(1+ \frac{1}{\sqrt{n}}\Big|\sum_{k\neq j}\frac{\lambda_k}{\lambda_k-\lambda_j}\Big|\Big).
\end{align*}
Similarly, by Theorem \ref{ThmTaylorEV} applied with $q=3$, the Cauchy-Schwarz inequality, the Minkowski inequality, and the inequality $|\operatorname{tr}(P_jER_jEP_j)|\leq \Vert P_jE|R_j|^{1/2}\Vert_2^2$ we have
\begin{align*}
&I_2+I_3\leq C\mathbb{E}^{1/4}\Vert P_jE|R_j|^{1/2}\Vert_2^8(\mathbb{E}\delta_j^4)^{1/4}+ \mathbb{E}^{1/4}\operatorname{tr}(P_jEP_j)^4\mathbb{P}^{1/4}(\delta_j\geq 1/4).
\end{align*}
By a moment computation and Lemmas \ref{KoltLounExp} and \ref{KoltLounConc}, we get
\begin{align*}
&I_2+I_3\leq C\frac{1}{n}\sum_{k\neq j}\frac{\lambda_j\lambda_k}{|\lambda_k-\lambda_j|}\sqrt{\frac{1}{n}\frac{\lambda_j}{g_j}\Big(\sum_{k\neq j}\frac{\lambda_k}{|\lambda_k-\lambda_j|}+\frac{\lambda_j}{g_j}\Big)}+Ce^{-c_2ng_j^2/\lambda_j^2}\frac{\lambda_j}{\sqrt{n}},
\end{align*}
where the second term is bounded by the first term. Now, inequality \eqref{EqEEVLowerBound} follows from the upper and lower bounds for $I_1$ and $I_2+I_3$  into \eqref{EqLowBound1}.
\end{proof}

\subsection{Empirical eigenprojections}
In this section, we will apply Theorem \ref{ThmTaylorSP} to the eigenprojections $\hat P_j$ of the empirical covariance operator $\hat\Sigma$. For this, we will extend Assumption \ref{SubGauss} slightly in order to be able to efficiently compute moments of polynomials chaos.
\begin{assumption}\label{SubGauss2}
For $j\geq 1$, let $\eta_j=\lambda_j^{-1/2}\langle X,u_j\rangle$ be the $j$-th Karhunen-Loève coefficient of $X$. Suppose that the $\eta_1,\eta_2,\dots$ are independent, symmetric, and sub-Gaussian, the latter meaning that there is a constant $\SGN$ such that
\[
\sup_{j\geq 1}\sup_{q\ge 1}q^{-1/2}\mathbb{E}^{1/q} |\eta_j|^{q} \le \SGN.
\]
\end{assumption}
Assumption \ref{SubGauss2} indeed implies Assumption \ref{SubGauss}, cf. \cite{V12}.

\begin{thm}\label{CorEEPBound}
Suppose that Assumption \ref{SubGauss2} holds. Then, for every $p\geq 1$, there are constants $c_1,C_1>0$ depending only on $\SGN$ and $p$ such that for every $j\geq 1$ satisfying \eqref{EqRGCond},
\begin{align*}
&\mathbb{E}^{1/2}\|\hat P_j-P_j\|_2^2\leq C_1\sqrt{\frac{1}{n}\sum_{k\neq j}\frac{\lambda_j\lambda_k}{(\lambda_k-\lambda_j)^2}}+C_1\Big(\frac{1}{n}\frac{\lambda_j}{g_j}\Big(\sum_{k\neq j}
\frac{\lambda_k}{|\lambda_k-\lambda_j|}+\frac{\lambda_j}{g_j}\Big)\Big)^{p/2}.
\end{align*}
\end{thm}

\begin{remark}
It is also possible to derive corresponding lower bounds by combining Theorem \ref{ThmTaylorSP} with the reverse triangular inequality.
\end{remark}

\begin{proof}
By Corollary \ref{CorImprovedBound2} and the Minkowski inequality, there is a constant $C>0$ depending only on $p$ such that
\begin{align*}
\mathbb{E}^{1/2}\|\hat P_j-P_j\|_2^2\leq C\sum_{m=1}^{p-1}\mathbb{E}^{1/2}\|(R_jE)^mP_j\|_2^2+C\mathbb{E}^{1/2}\delta_j^{2p}.
\end{align*}
Applying Lemma \ref{KoltLounExp} to the remainder term gives the remainder term in Theorem \ref{CorEEPBound}. Hence, the claim follows if we can show that for each $m\geq 1$,
\begin{align}\label{EqBoundPolChaos}
\mathbb{E}\|(R_jE)^mP_j\|_2^2\leq C_1\frac{1}{n}\sum_{k_1\neq j}\frac{\lambda_j\lambda_{k_1}}{(\lambda_{k_1}-\lambda_j)^2}\bigg(\frac{1}{n}\frac{\lambda_j}{g_j}\sum_{k\neq j}
\frac{\lambda_k}{|\lambda_k-\lambda_j|}\bigg)^{m-1}
\end{align}
for some constant $C_1>0$ depending only on $m$, provided that \eqref{EqRGCond} holds. From these inequalities, the claim follows from inserting \eqref{EqRGCond}. In what follows, let us fix $m\leq p-1$.
For $j,k\geq 1$ and $i\in\{1,\dots,n\}$, we set
\[
\eta_k^{(i)}=\lambda_k^{-1/2}\langle X_i,u_k\rangle\quad\text{ and }\quad \eta_{j,k}^{(i)}=\eta_j^{(i)}\eta_k^{(i)}-\delta_{j,k}.
\]
By Assumption \ref{SubGauss2}, the $\eta_k^{(i)}$, $k\geq 1$ and $i\in\{1,\dots,n\}$, are centered, independent, sub-Gaussian random variables. Now, we can write
\begin{align*}
&\|(R_jE)^mP_j\|_2^2
=\sum_{k_1\neq j}\frac{1}{(\lambda_{k_1}-\lambda_j)^2}\|P_{k_1}E(R_jE)^{m-1}P_j\|_2^2\\
&=\frac{1}{n^{2m}}\sum_{k_1\neq j}\frac{\lambda_j\lambda_{k_1}}{(\lambda_{k_1}-\lambda_j)^2}\Big(\sum_{k_2,\dots,k_m\neq j}\sum_{\mathbf{i}\in\{1,\dots,n\}^m}\Big(\prod_{a=2}^m \frac{\lambda_{k_a}}{\lambda_{k_a}-\lambda_j}\eta_{k_{a-1},k_a}^{(i_{a-1})}\Big)\eta_{k_m,j}^{(i_m)}\Big)^2.
\end{align*}
Multiplying out, we get
\begin{align}
&\mathbb{E}\|(R_jE)^mP_j\|_2^2\label{EqExpPertTerm}\\
&=\frac{1}{n^{2m}}\sum_{k_1,\dots,k_{2m-1}\neq j}\frac{\lambda_j\lambda_{k_1}}{(\lambda_{k_1}-\lambda_j)^2}\Big(\prod_{a=2}^{2m-1} \frac{\lambda_{k_a}}{\lambda_{k_a}-\lambda_j}\Big)\Big\{\sum_{\mathbf{i}\in\{1,\dots,n\}^{2m-1}}\mathbb{E}\eta_{\mathbf{k}}^{(\mathbf{i})}\Big\}\nonumber
\end{align}
with
\begin{align*}
&\eta_{\mathbf{k}}^{(\mathbf{i})}=\eta_{k_1,k_2}^{(i_1)}\dots\eta_{k_{m-1},k_m}^{(i_{m-1})}(\eta_{k_{m},k_{m+1}}^{(i_{m})}+\delta_{k_m,k_{m+1}})\eta_{k_{m+1},k_{m+2}}^{(i_{m+1})}\cdots\eta_{k_{2m-1},k_1}^{(i_{2m-1})}.
\end{align*}
To obtain \eqref{EqExpPertTerm}, we used that $\eta_{j}^{(i_m)}\eta_{j}^{(i_{m}')}$ is independent of the other Karhunen-Loève coefficients (since $j$ appears only twice) and that $\mathbb{E}\eta_{j}^{(i_m)}\eta_{j}^{(i_{m}')}=\delta_{i_m,i_{m}'}$, forcing $i_{m}'=i_{m}$.

Since \eqref{EqExpPertTerm} clearly implies \eqref{EqBoundPolChaos} in the case $m=1$, we restrict ourselves to $m>1$ in what follows. By Assumption \ref{SubGauss}, there is a constant $C_2$ depending only on $\SGN$ and $m$ such that $|\mathbb{E}\eta_{\mathbf{k}}^{(\mathbf{i})}|\leq C_2$ for every $\mathbf{i}\in\{1,\dots,n\}^{2m-1}$ and every $\mathbf{k}\in(\mathbb{N}\setminus \{j\})^{2m-1}$. Hence, in order to upper-bound \eqref{EqExpPertTerm} we have to bound for each $\mathbf{k}\in(\mathbb{N}\setminus \{j\})^{2m-1}$ the number of $(2m-1)$-tuples $\mathbf{i}\in\{1,\dots,n\}^{2m-1}$ for which $\mathbb{E}\eta_{\mathbf{k}}^{(\mathbf{i})}$ is non-zero.

First, in order that $\mathbb{E}\eta_{\mathbf{k}}^{(\mathbf{i})}\neq 0$, it is necessary that each number in $\mathbf{i}$ except of $i_m$ appears at least twice (use the independence of the Karhunen-Loève coefficients). It is easy to see that the number of $(2m-1)$-tuples $\mathbf{i}\in\{1,\dots,n\}^{2m}$ having the the latter property is bounded by $C_3n^{m}$ with $C_3$ depending only on $m$. Hence, we have
\begin{equation}\label{EqUpperBoundIndices1}
|\{\mathbf{i}\in\{1,\dots,n\}^{2m-1}:\mathbb{E}\eta_{\mathbf{k}}^{(\mathbf{i})}\neq 0\}|\leq C_3n^{m}
\end{equation}
for each $\mathbf{k}\in(\mathbb{N}\setminus \{j\})^{2m-1}$.

We now give another estimate for the cardinality in \eqref{EqUpperBoundIndices1}.  Fix $\mathbf{k}\in(\mathbb{N}\setminus \{j\})^{2m-1}$ and let $l=|\{k_a:a\in \{1,\dots,2m-1\}\}|$. Our goal is to show that
\begin{equation}\label{EqUpperBoundIndices2}
|\{\mathbf{i}\in\{1,\dots,n\}^{2m-1}:\mathbb{E}\eta_{\mathbf{k}}^{(\mathbf{i})}\neq 0\}|\leq C_4n^{2m-l}
\end{equation}
for some constant $C_4$ depending only on $m$. Since this is clear for $l=1$, we assume that $l\geq 2$ in what follows.

We call $a\in \{1,\dots, 2m-1\}$ a boundary point if $k_a\neq k_{a+1}$ (using the convention $k_{2m}=k_1$), and let $\mathcal{B}$ be the set of all boundary points. In what follows it is important to associate the  labels $k_a,k_{a+1}$ to each boundary point $a$ and to order all boundary points circularly such that each boundary point has exactly one matching label with both of its neighbor boundary points. We now also fix $\mathbf{i}\in\{1,\dots,n\}^{2m-1}$ such that $\mathbb{E}\eta_{\mathbf{k}}^{(\mathbf{i})}\neq 0$ and claim that
\begin{equation}\label{EqCounting}
\mathbb{E}\prod_{b\in\mathcal{B}}\eta_{k_b,k_{b+1}}^{(i_b)}\neq 0\quad\text{ and thus }\quad |\{i_b:b\in \mathcal{B}\}|\leq |\mathcal{B}|-l+1.
\end{equation}
The first claim follows from the independence and symmetry of the Karhunen-Loève coefficients, the second claim is proved below by induction. Using \eqref{EqCounting}, we get that $|\{i_a:a\in \{1,\dots, 2m-1\}\}|\leq 2m-l$. Hence, we conclude that each $\mathbf{i}$ from the set in \eqref{EqUpperBoundIndices2} has property that there are at most $2m-l$ different entries, from which \eqref{EqUpperBoundIndices2} follows.

It remains to deduce the second claim in \eqref{EqCounting} from the first one. This can be done by induction on $|\mathcal{B}|$. For $|\mathcal{B}|=2$ and $l=2$ the claim is clear. For the induction step assume that the implication in \eqref{EqCounting} holds for all sets $\mathcal{B}'$ of boundary points with $|\mathcal{B}'|<|\mathcal{B}|$. If $|\mathcal{B}|=l$, then the claim follows because in this case $i_a=i_{b}$ for all neighbor boundary points $a,b\in\mathcal{B}$ (use the independence of the Karhunen-Loève coefficients and the fact that for $a,b\in\mathcal{B}$ we have either $k_{a+1}=k_b$ or $k_{a}=k_{b+1}$ if and only if $a,b$ are neighbors) and thus $|\{i_b:b\in \mathcal{B}\}|= 1$. If $|\mathcal{B}|<l$, then consider $\mathcal{B}_1=\{b\in \mathcal{B}:i_b=i_a\}$ for some $a\in \mathcal{B}$. Since the claim is clear for $\mathcal{B}_1=\mathcal{B}$, we restrict ourselves to $\mathcal{B}_1\subset\mathcal{B}$. Our goal is to apply the induction hypothesis to an appropriate partition of $\mathcal{B}\setminus\mathcal{B}_1$. For this we use the fact that for each $a\in \mathcal{B}_1$, there is a $a\neq b\in \mathcal{B}_1$ with $k_a=k_b$ or $k_a=k_{b+1}$, as well as a  $a\neq c\in \mathcal{B}_1$ with $k_{a+1}=k_c$ or $k_{a+1}=k_{c+1}$ (this follows from the independence of the Karhunen-Loève coefficients using that the expectation in \eqref{EqCounting} is non-zero). Additionally, we write $|\mathcal{B}_1|=e_1+f_1$ where $e_1$ is the number of boundary points $b\in \mathcal{B}_1$ for which $k_b\neq k_c$ for all boundary points $c\neq b$. Using the previous facts, we can construct a partition $\mathcal{B}_2,\dots,\mathcal{B}_f$, $f\leq f_1$, of $\mathcal{B}\setminus\mathcal{B}_1$ such that the $\mathcal{B}_j$ have pairwise different $k$-labels and such that the points in $\mathcal{B}_j$ can be ordered (circularly) in such a way that each boundary point has one matching label with both its neighbors. In particular, we have $\mathbb{E}\prod_{b\in\mathcal{B}_j}\eta_{k_b,k_{b+1}}^{(i_b)}\neq 0$ and thus by the induction hypothesis $|\{i_b:b\in \mathcal{B}_j\}|\leq |\mathcal{B}_j|-|\{k_b:b\in \mathcal{B}_j\}|+1$. Using that $\sum_{j=2}^f|\{k_b:b\in \mathcal{B}_j\}|=l-e_1$ and $\sum_{j=2}^f|\mathcal{B}_j|=|\mathcal{B}|-e_1-f_1$, we conclude that $|\{i_b:b\in \mathcal{B}\}|\leq 1+\sum_{j=2}^f|\{i_b:b\in \mathcal{B}_j\}|\leq 1+|\mathcal{B}|-e_1-f_1-l+e_1+f$, which gives the claim.

Inserting \eqref{EqUpperBoundIndices1} and \eqref{EqUpperBoundIndices2} into \eqref{EqExpPertTerm}, we conclude that
\begin{align*}
&\mathbb{E}\|(R_jE)^mP_j\|_2^2\\
&\leq C_5\sum_{\mathcal{I}_1,\dots,\mathcal{I}_l}\frac{1}{n^{m\vee l}}\sum_{k_1,\dots,k_{l}\neq j}\frac{\lambda_j\lambda_{k_1}^{|\mathcal{I}_1|}}{|\lambda_{k_1}-\lambda_j|^{|\mathcal{I}_1|+1}}\frac{\lambda_{k_2}^{|\mathcal{I}_2|}}{|\lambda_{k_2}-\lambda_j|^{|\mathcal{I}_2|}}\cdots\frac{\lambda_{k_l}^{|\mathcal{I}_l|}}{|\lambda_{k_l}-\lambda_j|^{|\mathcal{I}_l|}},
\end{align*}
where the sum is over all partitions of $\{1,\dots,2m-1\}$ with $1\in \mathcal{I}_1$. For simplicity, we now focus on the case $|\mathcal{I}_1|=1$, the remaining cases follow analogously. Furthermore, we consider separately the cases $m\leq l$ and $l< m$. First, for $m\leq  l$, we have
\begin{align*}
&\frac{1}{n^l}\sum_{k_1,\dots,k_{l}\neq j}\frac{\lambda_j\lambda_{k_1}}{(\lambda_{k_1}-\lambda_j)^{2}}\frac{\lambda_{k_2}^{|\mathcal{I}_2|}}{|\lambda_{k_2}-\lambda_j|^{|\mathcal{I}_2|}}\cdots\frac{\lambda_{k_l}^{|\mathcal{I}_l|}}{|\lambda_{k_l}-\lambda_j|^{|\mathcal{I}_l|}}\\
&\leq \frac{1}{n^l}\sum_{k_1\neq j}\frac{\lambda_j\lambda_{k_1}}{(\lambda_{k_1}-\lambda_j)^{2}}\bigg(\frac{2\lambda_j}{g_j}\bigg)^{2m-l-1}\bigg(\sum_{k\neq j}\frac{\lambda_{k}}{|\lambda_{k}-\lambda_j|}\bigg)^{l-1}\\
&\leq \frac{2^{l-1}}{n}\sum_{k_1\neq j}\frac{\lambda_j\lambda_{k_1}}{(\lambda_{k_1}-\lambda_j)^{2}}\bigg(\frac{1}{n}\frac{\lambda_j}{g_j}\sum_{k\neq j}\frac{\lambda_{k}}{|\lambda_{k}-\lambda_j|}\bigg)^{m-1},
\end{align*}
where we applied $\max_{k\neq j}\lambda_k/|\lambda_k-\lambda_j|\leq 2\lambda_j/g_j$, the bound $\sum_{c=2}^l(|\mathcal{I}_c|-1)=2m-2-l+1\leq l-1$, and \eqref{EqRGCond} with $c_1<2$.
On the other hand, if $l< m$, meaning that $\sum_{c=2}^l(|\mathcal{I}_c|-1)=2m-l-1> l$, then the number of $c$'s such that $|\mathcal{I}_c|=1$ is smaller than the number of $c$'s such that $|\mathcal{I}_c|>2$. Moreover, for $|\mathcal{I}_c|>2$, we can bound
\[
\sum_{k\neq j}\frac{\lambda_{k}^{|\mathcal{I}_c|}}{|\lambda_{k}-\lambda_j|^{|\mathcal{I}_c|}}\leq \bigg(\frac{2\lambda_j}{g_j}\sum_{k\neq j}\frac{\lambda_{k}}{|\lambda_{k}-\lambda_j|}\bigg)^{\frac{|\mathcal{I}_c|}{2}}\bigwedge\frac{2\lambda_j}{g_j}\bigg(\frac{2\lambda_j}{g_j}\sum_{k\neq j}\frac{\lambda_{k}}{|\lambda_{k}-\lambda_j|}\bigg)^{\frac{|\mathcal{I}_c|-1}{2}},
\]
and we get
\begin{align*}
&\frac{1}{n^{m}}\sum_{k_1,\dots,k_{l}\neq j}\frac{\lambda_j\lambda_{k_1}}{(\lambda_{k_1}-\lambda_j)^{2}}\frac{\lambda_{k_2}^{|\mathcal{I}_2|}}{|\lambda_{k_2}-\lambda_j|^{|\mathcal{I}_2|}}\cdots\frac{\lambda_{k_l}^{|\mathcal{I}_l|}}{|\lambda_{k_l}-\lambda_j|^{|\mathcal{I}_l|}}\\
&\leq \frac{1}{n}\sum_{k_1\neq j}\frac{\lambda_j\lambda_{k_1}}{(\lambda_{k_1}-\lambda_j)^{2}}\bigg(\frac{2}{n}\frac{\lambda_j}{g_j}\sum_{k\neq j}\frac{\lambda_{k}}{|\lambda_{k}-\lambda_j|}\bigg)^{m-1}.
\end{align*}
This completes the proof.
\end{proof}

\subsection{Sub-exponential decay of eigenvalues}
Let us briefly specialize our findings to sub-exponential decay of eigenvalues of the form
\begin{equation}\label{EqExpDec}
\lambda_j=e^{- j^\alpha},\quad j\ge 1,
\end{equation}
for some $\alpha\in (0,1]$. In this case, the eigenvalue expressions in Theorems \ref{CorEEVUpperBound} and \ref{CorEEPBound} can be bounded as follows.
\begin{lemma}\label{LemEvExpr} If \eqref{EqExpDec} holds for some $\alpha\in (0,1]$, then there is a constant $C>1$ depending only on $\alpha$ such that, for every $j\geq 1$,
\[
\frac{\lambda_j}{g_j}\leq Cj^{1-\alpha},\quad\sum_{k\neq j}\frac{\lambda_k}{|\lambda_k-\lambda_j|}\leq Cj,\quad\sum_{k\neq j}\frac{\lambda_j\lambda_k}{(\lambda_k-\lambda_j)^2}\leq Cj^{2-2\alpha}.
\]
Moreover, for every $j\geq C$,
\[
\sum_{k\neq j}\frac{\lambda_k}{\lambda_k-\lambda_j}\geq C^{-1}j.
\]
\end{lemma}
For $\alpha=1$, the claim follows from the bound $e^{-j}-e^{k}\geq (1-e^{-1})e^{-j}$, $k>j$, in combination with the inequalities $\sum_{k>j}e^{-k}\leq Ce^{-j}$ and $\sum_{k<j}e^{k}\leq Ce^{j}$, $j\geq 1$. For $\alpha<1$, the claim follows from similar concavity arguments combined with a comparison of the sums with an integral and estimates for the incomplete Gamma function. We omit the details of the proof. Inserting Lemma \ref{LemEvExpr} into Theorem \ref{CorEEVUpperBound}, we have the following consequence (note that the remainder term in \eqref{EqEEVUpperBound} can be dropped since $e^{-c_2nj^{2\alpha-2}}\leq e^{- j^\alpha}$ for $j\leq c_1n^{1/(2-\alpha)}$ with $c_1$ small enough).
\begin{corollary}\label{CorEVExpDecay}
Suppose that Assumption \ref{SubGauss} holds and that \eqref{EqExpDec} holds for some $\alpha\in (0,1]$. Then there are constants $c_1,C_1$ depending only on $\SGN$ and $\alpha$ such that for all $c_1^{-1}\leq j\leq c_1n^{1/(2-\alpha)}$, 
\[ C_1^{-1}\Big(\frac{1}{\sqrt{n}}+\frac{j}{n}\Big)\leq \mathbb{E}^{1/2}(\hat\lambda_j/\lambda_j-1)^2\leq C_1\Big(\frac{1}{\sqrt{n}}+\frac{j}{n}\Big).
\]
\end{corollary}
Corollary~\ref{CorEVExpDecay} provides matching upper and lower bounds in $L^2$-norm in the range $j\leq cn^{1/(2-\alpha)}$. These bounds reveal a sharp phase transition. In fact, the linear perturbation term dominates the bound for $j\leq c_2\sqrt{n}$, while the second order perturbation dominates the bound for $c_2\sqrt{n} \leq j\leq c_1n^{1/(2-\alpha)}$. Interestingly, in the latter case one still has an eigenvalue separation property, and it turns out that the Hilbert-Schmidt distance of perturbed and unperturbed eigenprojections are dominated by the linear perturbation term almost throughout the (optimal) range $j\leq c_1n^{1/(2-\alpha)}$:
\begin{corollary}
Suppose that Assumption \ref{SubGauss2} holds and that \eqref{EqExpDec} holds for some $\alpha\in (0,1]$. Let $\epsilon>0$. Then there are constants $c_1,C_1$ depending only on $\SGN$, $\alpha$, and $\epsilon$ such that for every $j\leq c_1n^{(1-\epsilon)/(2-\alpha)}$,
\[
\mathbb{E}^{1/2}\|\hat P_j-P_j\|_2^2\leq C_1 \frac{j^{1-\alpha}}{\sqrt{n}}.
\]
\end{corollary}
\begin{remark}
Similar results hold if $\lambda_j=e^{- cj^\alpha}$, $j\ge 1$. One-sided versions of such an eigenvalue behavior arise for a large class of operators defined by a kernel; see e.g. \cite{B18,MR2586970} and Section \ref{SecDisc} below.
\end{remark}

\subsection{On the relative rank condition from \cite{JW18a,JW18b}}\label{SecRRC}
In case of the empirical covariance operator our perturbation results can be applied under the condition
\begin{equation}\label{EqCondKLN}
\frac{\lambda_j}{g_j}\Big(\sum_{k\neq j}\frac{\lambda_k}{|\lambda_k-\lambda_j|}+\frac{\lambda_j}{g_j}\Big)\leq cn.
\end{equation}
For instance, in the case of exponentially decaying eigenvalues, \eqref{EqCondKLN} means that $j\leq cn$. Hence, our framework allows us to study empirical eigenvalues and eigenprojections in a (nearly) optimal range (note that $\hat\Sigma$ is of rank $n$, hence all eigenprojections with index larger $n$ are non-unique). More generally, we conjecture that if \eqref{EqCondKLN} does not hold (with $c$ large enough), then we do not have the eigenvalue separation property from Remark \ref{RemEVSep}, in which case we could not even cluster empirical and population eigenvalues and eigenprojections appropriately.

Improving standard perturbation results in the case of the empirical covariance operator has been considered previsouly in Jirak and Wahl \cite{JW18a,JW18b}, who established relative perturbation bounds, tailored for empirical covariance operators under the condition
\begin{equation}\label{EqCondRRCond}
\sum_{k\neq j}\frac{\lambda_k}{|\lambda_k-\lambda_j|}+\frac{\lambda_j}{g_j}\leq c\sqrt{n}.
\end{equation}
They showed that under \eqref{EqCondRRCond} a strong contraction property holds, implying that the difference of empirical and true eigenvalues (resp. eigenprojections) can be accurately approximated by the first order perturbation terms. The achievement of this section is to extend \eqref{EqCondRRCond} to \eqref{EqCondKLN}, by invoking higher-order expansions based on $\delta_j$. This extended regime reveals new features due to the fact that the contraction property does not continue to hold.

In the case of empirical eigenvalues, we derive upper and lower bounds in $L^2$-norm, dominated by the first order perturbation term if \eqref{EqCondRRCond} holds, and dominated by the second order perturbation term if \eqref{EqCondRRCond} does not hold but \eqref{EqCondKLN} holds. \cite{JW18a} established the necessity of \eqref{EqCondRRCond} (for an accurate first order perturbation expansion) by constructing a counterexample in terms of an one-factor model. In contrast, our results extend the necessity of \eqref{EqCondRRCond} to a much larger class of models.

The situation is different in the case of empirical eigenprojections,  where Theorem \ref{CorEEPBound} shows that the Hilbert-Schmidt distance between true and empirical eigenprojections is dominated by the linear perturbation term throughout \eqref{EqCondKLN}. Interestingly, it follows from Lemma \ref{LemEVSeparation} that if \eqref{EqRGCond} holds, then \eqref{EqEEVUpperBound} yields $\mathbb{E}^{1/2}(\hat\lambda_j-\lambda_j)^2<g_j/2+remainder$, meaning that the minimum of the $L^2$-distance between $\hat\lambda_j$ and the $(\lambda_k)$ is attained at $k=j$ (ignoring the remainder term). This eigenvalue separation property gives an explanation for the strong result in Theorem \ref{CorEEPBound}. We believe that the existence of this extended range is closely related to our strong probabilistic assumptions (independent and sub-Gaussian Karhunen-Loève coefficients). In fact, in the one-factor model constructed in \cite{JW18a}, Condition \eqref{EqCondRRCond} is also equivalent to a weak form of separation of eigenvalues, indicating that Theorem \ref{CorEEPBound} does not continue to hold under the weaker moment assumptions from \cite{JW18a}.

\subsection{Extensions}\label{SecDisc}
\subsubsection*{Kernel operators and kernel Gram matrices}
Kernel operators and their approximations by kernel Gram matrices play a fundamental role in machine learning problems. While we discussed applications to the empirical covariance operator, we show in this section how these results can be transferred to kernel operators; see e.g. \cite{CS02,ZBB04,VRCGO05} for more details. For this, let $k(\cdot,\cdot)$ be a continuous and positive definite kernel on a compact mertric space $\mathcal{X}$, and let $\mathcal{H}$ be the reproducing kernel Hilbert space (RKHS) of $k$. Given a probability measure $\rho$ on $\mathcal{X}$, we can define the integral operator $K_\rho:L^2(\rho)\rightarrow L^2(\rho), K_\rho f(x)= \int_\mathcal{X} k(x,y)f(y)\rho(dy)$. It is easy to see, that $K_\rho$ is a self-adjoint positive trace-class operator. Moreover, given independent random variables $X,X_1,\dots,X_n$ in $\mathcal{X}$ with common distribution $\rho$, we can construct the approximation $K_n=(n^{-1}k(X_i,X_j))_{i,j=1}^n$ of $K_\rho$. The close link to covariance operators can be seen by introducing the so-called restriction operator $R_\rho:\mathcal{H}\rightarrow L^2(\rho)$, mapping $f\in\mathcal{H}$ to $f\in L^2(\rho)$ by restricting it to the support of $\rho$.
\[
R_\rho R_\rho^*=K_\rho,\qquad R_\rho^*R_\rho=\mathbb{E}k(X,\cdot)\otimes k(X,\cdot)=\Sigma.
\]
Similarly, $R_n:\mathcal{H}\rightarrow\mathbb{R}^n,f\mapsto (f(X_1),\dots,f(X_n))^T$, leads to
\[
R_n R_n^*=K_n,\qquad R_n^*R_n=\frac{1}{n}\sum_{i=1}^nk(X_i,\cdot)\otimes k(X_i,\cdot)=\Sigma_n.
\]
This correspondence readily allows to transfer perturbation problems for eigenvalues (and eigenprojections) of $K_\rho$ and $K_n$ to analogous problems for $\Sigma$ and $\hat\Sigma$.


\subsubsection*{High-probability bounds}
While Theorems \ref{CorEEVUpperBound} and \ref{CorEEPBound} establish bounds in expectation, similar results can be derived for the $L^p$-norm and with high probability. This can be done using moment estimates and concentration inequalities for polynomials in independent sub-Gaussian random variables derived in \cite{L06,AW15,ALM18}. In the case of eigenprojections, however, this leads to lengthy expressions, as we have to compute many intricate norms. Interestingly, the random variable $\|\hat P_j-P_j\|_2^2$ is itself bounded, meaning that higher-order norms ultimately play only a minor role.

\bibliographystyle{plain}
\bibliography{lit}

\appendix
\section{Appendix}
The proof of our main result is based on a Taylor expansion with explicit remainder term. In this appendix, we will discuss how alternative approaches based on complex analytic arguments can be used.

\subsection{Rellich's perturbation theorem and its consequences}
In this section, we provide an alternative argument to validate the perturbation series from Corollary \ref{CorPertSeries} based on Rellich's theorem, see e.g. \cite{W86}. For simplicity, we assume that $\mathcal{H}=\mathbb{R}^d$. Let $j\geq 1$ be such that $g_j>0$.

For $z\in \mathbb{C}$, set $\Sigma(z)=\Sigma+z(\hat\Sigma-\Sigma)$ and $E(z)=\Sigma(z)-\Sigma=z(\hat\Sigma-\Sigma)$. We first apply Rellich's theorem, saying that there is an open set $\Omega=\bar{\Omega}\subseteq\mathbb{C}$ with $[0,1]\subseteq \Omega$ and holomorphic functions $\lambda_1,\dots,\lambda_d:\Omega\rightarrow \mathbb{C}$ and $u_1,\dots,u_d:\Omega\rightarrow \mathbb{C}^n$ such that $u_k^*u_l=\delta_{kl}$ and
\begin{equation}\label{EqSDholomorphic}
\Sigma(z)=\sum_{k=1}^d\lambda_j(z)P_k(z),\qquad P_j(z)=u_j(z)u_j^*(z)
\end{equation}
for all $z\in \Omega$, where $u_j^*(z)=\overline{u_j(\bar{z})}^T$. Since $u_j^*$ is holomorphic, we get that $P_j(z)$ is holomorphic. We suppose $\lambda_1(0)\geq\dots \geq \lambda_d(0)$ such that $\lambda_j(0)=\lambda_j$ and the $P_j(0)=P_j$ by uniqueness.

Similarly as in the proof of Lemma \ref{LemEVSeparation}, for every $z\in[0,1]$, we have
\begin{equation}\label{EqEVSep3}
|\lambda_j(z)-\lambda_j|\leq g_j/4
\end{equation}
as well as
\begin{equation}\label{EqEVSep4}
\lambda_{j+1}(z)-\lambda_{j+1}\leq (\lambda_j-\lambda_{j+1})/4,\quad \lambda_{j-1}(z)-\lambda_{j-1}\geq (\lambda_{j-1}-\lambda_{j})/4
\end{equation}


We now conclude the proof, by combining the previous steps with Lemma~\ref{LemCoeffBound}. By \eqref{EqEVSep3} and \eqref{EqEVSep4}, we conclude that for every $z\in[0,1]$, $\lambda_j(z)$ is a simple eigenvalue and it is the $j$-th largest eigenvalue of $\Sigma(z)$. Moreover, $P_j(z)$ is the corresponding spectral projector. In particular, we have $\hat\lambda_j=\lambda_j(1)$ and $\hat P_j=P_j(1)$. Since $P_j$ is holomorphic, it has a series representation near $0$. Moreover, by \cite{K95,C83}, it is given by
\begin{equation}\label{EqClassicalPertSeries}
\lambda_j(z)=\sum_{p\geq 0}z^p\lambda_j^{(p)},\qquad P_j(z)=\sum_{p\geq 0}z^pP_j^{(p)}.
\end{equation}
By Lemma \ref{LemCoeffBound}, both series converge absolutely (in Hilbert-Schmidt norm) in a region containing $[0,1]$. By possibly shrinking $\Omega$, \eqref{EqSDholomorphic} and \eqref{EqClassicalPertSeries} hold for all $z\in\Omega$ with $[0,1]\subseteq \Omega$. Since both sides of the equations are holomorphic in $\Omega$ and coincide in a small neighborhood of zero, they coincide by the uniqueness property of holomorphic functions. In particular, we have
\begin{equation}
\hat \lambda_j=\lambda_j(1)=\sum_{p\geq 0}\lambda_j^{(p)},\qquad \hat P_j=P_j(1)=\sum_{p\geq 0}P_j^{(p)}.
\end{equation}
This completes the proof.\qed

\subsection{Holomorphic functional calculus and its consequences}\label{SecHFC}

Another powerful machinery to derive perturbation bounds is given by the holomorphic functional calculus for linear operators, see e.g. Kato \cite{K95} and Chatelin \cite{C83}. Combining the holomorphic functional calculus with the eigenvalue separation in Lemma \ref{LemEVSeparation}, we get the following version of Theorem \ref{ThmTaylorSP}.
\begin{corollary}\label{CorEPTaylorExpHFC}
Suppose that $\delta_j<1/2$. Then, for each $p\geq 1$, we have
\begin{equation*}
\|\hat P_j-\sum_{n=0}^{p-1}P_j^{(n)}\|_2\leq  2\frac{(2\delta_j)^{p}}{1-2\delta_j}.
\end{equation*}
\end{corollary}
A similar result can be obtained in the case of eigenvalues. Due to some technical obstacles, the following bound includes a constant $C$ depending on the dimension $d$ of $\mathcal{H}$. We conjecture that the bound holds for $C=1/2$.
\begin{corollary}\label{CorEVTaylorExpHFC}
Suppose that $\delta_j<1/2$. Then, for each $p\geq 1$, we have
\begin{equation*}
|\hat \lambda_j-\sum_{n=0}^{p-1}\lambda_j^{(n)}|\leq Cg_j \frac{(2\delta_j)^{p}}{1-2\delta_j}.
\end{equation*}
where $C=1$ for $p=1$ and $C$ is a constant depending on $d$ otherwise.
\end{corollary}

We conclude that the perturbation series in Corollary \ref{CorPertSeries} also hold under the slightly weaker condition $\delta_j<1/2$. Note that this seems to be difficult to reach with our approach based on explicit remainder terms, in which case the natural condition is $\delta_j'<1/4$. Conversely, the bounds in Corollaries \ref{CorEPTaylorExpHFC} and \ref{CorEVTaylorExpHFC} are weaker than those in Theorems \ref{ThmTaylorSP} and \ref{ThmTaylorEV}. For instance, Corollary \ref{ThmTaylorEV} contains the factor $\Vert P_jE|R_j|^{1/2}\Vert_2^2$ instead of $g_j\delta_j^2$. In the special case  $p=1$, Corollary \ref{CorEVTaylorExpHFC} gives the bound $Cg_j\delta_j$, while Corollary \ref{PertBoundEV} provides the size of the first-order and second-order perturbation terms. In fact, we believe that our explicit approach has several advantages over the holomorphic functional calculus, since it is possible to deal more directly with the terms of the perturbation series. For instance, it seems out of reach to obtain Corollary \ref{CorImprovedBound3} (and the other consequences from Sections \ref{SecPertSeries} and \ref{SecTightPB}) using the holomorphic functional calculus.

\begin{proof}[Proof of Corollary \ref{CorEPTaylorExpHFC}]
The starting point is the Cauchy integral formula for spectral projectors using the notion of the resolvent. In fact, if $\lambda_j$ is simple, then we have
\[
P_j=-\frac{1}{2\pi i}\oint_{\gamma_j}\limits(\Sigma-zI)^{-1}\,dz
\]
with circle $\gamma_j=\{z:|z-\lambda_j|=g_j/2\}$ enclosing only $\lambda_j$ counterclockwise. The second main ingredient is Lemma \ref{LemEVSeparation}, stating that under the condition $\delta_j<1/2$, the circle $\gamma_j$ encloses $\hat\lambda_j$, while the remaining empirical eigenvalues lie (strictly) outside. Thus we also have
\[
\hat P_j=-\frac{1}{2\pi i}\oint_{\gamma_j}\limits(\hat \Sigma-zI)^{-1}\,dz.
\]
We now apply (formally) the second von Neumann series and verify its validity afterwards
\begin{equation}\label{Eq2NeumannSeries}
(\hat \Sigma-zI)^{-1}-(\Sigma-zI)^{-1}=\sum_{n\geq 1}(-1)^n(\Sigma-zI)^{-1}(E( \Sigma-zI)^{-1})^n.
\end{equation}
Inserting $|R_j|^{1/2}+g_j^{-1/2}P_j$ and its inverse appropriately, we have
\begin{align*}
&(\Sigma-zI)^{-1}(E( \Sigma-zI)^{-1})^n\nonumber\\
&=(\Sigma-zI)^{-1}(|R_j|^{-1/2}+g_j^{1/2}P_j)\\
&\ \ \ \times\Big((|R_j|^{1/2}+g_j^{-1/2}P_j)E(|R_j|^{1/2}+g_j^{-1/2}P_j)\\
&\ \ \ \ \ \ \ \times(|R_j|^{-1/2}+g_j^{1/2}P_j)(\Sigma-zI)^{-1}(|R_j|^{-1/2}+g_j^{1/2}P_j)\Big)^n\\
&\ \ \ \times(|R_j|^{1/2}+g_j^{-1/2}P_j)E(|R_j|^{1/2}+g_j^{-1/2}P_j)\\
&\ \ \ \times(|R_j|^{-1/2}+g_j^{1/2}P_j)(\Sigma-zI)^{-1}.
\end{align*}
Hence, for $z\in\gamma_j$,
\begin{align}
&\|(\Sigma-zI)^{-1}(E( \Sigma-zI)^{-1})^n\|_{\infty}\nonumber\\
&\leq \|(\Sigma-zI)^{-1}(|R_j|^{-1/2}+g_j^{1/2}P_j)\|_\infty^2\nonumber\\
&\ \ \ \times\|(|R_j|^{-1/2}+g_j^{1/2}P_j)(\Sigma-zI)^{-1}(|R_j|^{-1/2}+g_j^{1/2}P_j)\|_\infty^{n}\delta_j^n\nonumber\\
&\leq 4g_j^{-1}(2\delta_j)^n.\label{EqBoundResExpr}
\end{align}
We conclude that the right-hand side of \eqref{Eq2NeumannSeries} converges absolutely, from which we deduce the identity \eqref{Eq2NeumannSeries} e.g. by analytic continuation. By the residue theorem (cf. \cite{K95,C83}), we have
\[
P_j^{(n)}=\frac{(-1)^{n-1}}{2\pi i}\oint_{\gamma_j}(\Sigma-zI)^{-1}(E( \Sigma-zI)^{-1})^n\,dz,
\]
and we conclude, using \eqref{EqBoundResExpr},
\begin{align*}
&\|\hat P_j-\sum_{n=0}^{p-1}P_j^{(n)}\|_\infty\\&\leq \sum_{n\geq p}\frac{1}{2\pi}\oint_{\gamma_j}\limits\|(\Sigma-zI)^{-1}(E( \Sigma-zI)^{-1})^n\|_{\infty}\,dz\leq \sum_{n\geq p}(2\delta_j)^n.
\end{align*}
This completes the proof.
\end{proof}

\begin{proof}[Proof of Corollary \ref{CorEVTaylorExpHFC}]
This time, the starting point is
\begin{align}\label{EqEvHFC}
(\hat\lambda_j-\lambda_j)\hat P_j=-\frac{1}{2\pi i}\oint_{\gamma_j}(z-\lambda_j)(\hat \Sigma-zI)^{-1}\,dz.
\end{align}
First, taking the operator norm, and applying \eqref{Eq2NeumannSeries} and \eqref{EqBoundResExpr}, we get $
|\hat\lambda_j-\lambda_j|\leq g_j\delta_j/(1-2g_j)$. This yields the claim in the case $p=1$. For $p\geq 2$, we have to apply the trace instead of the operator norm (since the left hand side of \eqref{EqEvHFC} contains $\hat P_j$, while the right-hand side can be expanded in terms of the eigenprojections of $\Sigma$). First, by proceeding similarly as in \eqref{EqBoundResExpr}, we have
\[
\|(z-\lambda_j)(\Sigma-zI)^{-1}(E( \Sigma-zI)^{-1})^n\|_\infty\leq 2(2\delta_j)^n,\quad z\in\gamma_j.
\]
By the residue theorem, we have
\[
\lambda_j^{(n)}=\frac{(-1)^{n-1}}{2\pi i}\operatorname{tr}\oint_{\gamma_j}(z-\lambda_j)(\Sigma-zI)^{-1}(E( \Sigma-zI)^{-1})^n\,dz.
\]
and we conclude
\begin{align*}
&|\hat \lambda_j-\sum_{n=0}^{p-1}\lambda_j^{(n)}|\\
&\leq \sum_{n\geq p}\frac{1}{2\pi}|\operatorname{tr}\oint_{\gamma_j}(z-\lambda_j)(\Sigma-zI)^{-1}(E( \Sigma-zI)^{-1})^n\,dz|\\
&\leq \sum_{n\geq p}\frac{2d}{2\pi}\|\oint_{\gamma_j}(z-\lambda_j)(\Sigma-zI)^{-1}(E( \Sigma-zI)^{-1})^n\,dz\|_\infty\leq 2dg_j\sum_{n\geq p}(2\delta_j)^n.
\end{align*}
This completes the proof.
\end{proof}



\end{document}